\documentclass[12pt]{amsart}
\usepackage{
  fullpage,
  graphicx,amssymb,epic,eepic,cancel,picins,tikz,
color,hyphenat,
  multicol, 
slashbox, 
}
\usepackage[all]{xy}

\newtheorem{theorem}{Theorem}[section]
\newtheorem{lemma}[theorem]{Lemma}
\newtheorem{proposition}[theorem]{Proposition}

\newtheorem{definition}{Definition}
\newtheorem{remark}{Remark}
\newtheorem{example}{Example}

\makeatletter
\newcommand{\extp}{\@ifnextchar^\@extp{\@extp^{\,}}}
\def\@extp^#1{\mathop{\mathchoice{\textstyle}{}{}{}\bigwedge\nolimits^{\!#1\!}}}
\makeatother

\def\calV{{\mathcal V}}

\def\bbR{{\mathbb R}}

\def\cal{\mathcal}
\def\bb{\mathbb}

\def\dim{\operatorname{dim}}

\def\k{\Bbbk}

\begin{document}
\newdimen\captionwidth\captionwidth=\hsize

\title{Odd Khovanov homology for hyperplane arrangements}

\author{Zsuzsanna Dancso and Anthony Licata}
\thanks{Anthony Licata was partially supported by the National Science Foundation under agreement No. DMS-0635607. Any opinion, finding, and conclusions or recommendations expressed in the material do not necessarily reflect the views of the National Science Foundation.}
\address{Zsuzsanna Dancso\\
Department of Mathematics\\
  University of Toronto\\
  40 St George St\\
  Toronto Ontario M5S 2E4\\
  Canada
}
\email{zsuzsi@math.toronto.edu}
\urladdr{http://www.math.toronto.edu/zsuzsi}
\address{Anthony Licata\\
Mathematical Sciences Institute\\
  Australian National University\\
  Union Lane\\
  Canberra, ACT, 0200\\
  Australia
}
\email{amlicata@gmail.com}
\urladdr{http://maths-people.anu.edu.au/~licatat}

\date{January 23, 2011}

\begin{abstract}
We define several homology theories for central hyperplane arrangements, categorifying well-known polynomial invariants including the characteristic polynomial, 
Poincar\'e polynomial, and Tutte polynomial.  We consider basic algebraic properties of such chain complexes, including long-exact sequences associated to 
deletion-restriction triples and dg-algebra structures.  We also consider signed hyperplane arrangements, and generalize the odd Khovanov homology of 
Ozsv\'ath-Rasmussen-Szab\'o from link projections to signed arrangements.  We define hyperplane Reidemeister moves which generalize the usual Reidemeister moves 
from framed link projections to signed arrangements, and prove that the chain homotopy type associated to a signed arrangement is invariant under hyperplane Reidemeister moves.
\end{abstract}

\maketitle
\tableofcontents

\section{Introduction}
Let $\k$ be a field, and let $V$ be a vector space over $\k$ endowed with an inner product $\langle-,-\rangle:
V\times V\rightarrow \k$. For convenience, we require $\k$ to be of characteristic zero so that the inner product induces an isomorphism
between $V$ and $V^*$. (This restriction can be removed by rephrasing all that follows in terms of dual spaces.)
A \emph{vector arrangement} is a collection of vectors $\nu_1,...,\nu_n$ in $V$.  A vector arrangement determines an arrangement of 
hyperplanes 
$\cal H = \{V; H_1,...,H_n\}$ in $V$, with $H_i = \nu_i^\perp=\{v \in V: \langle v, \nu_i \rangle=0\}$.  We allow the degenerate case $\nu_i=0$, in which case $H_i=V$.
For $S \subset [n]=\{1,\hdots,n\}$, set
$H_S = \cap_{s\in S} H_s.$  Important features of the hyperplane arrangement $\cal H$ are captured by polynomial invariants associated to the arrangement; 
here ``polynomial invariant" 
means a polynomial which depends only on the associated matroid, that is, only on the linear 
dependencies between hyperplanes. An example is the \emph{characteristic polynomial} $\chi(\cal H)$ of $\cal H$,
$$
\chi(\cal H)= \sum_{S \subseteq [n]} (-1)^{|S|} t^{\dim H_S},
$$
which generalizes the chromatic polynomial of a graph, \cite{TeraoOrlik}.  An important feature of the characteristic polynomial is that it satisfies a 
deletion-restriction relation,
\begin{equation}\label{eq:delres0}
 \chi(\cal H)=\chi(\cal H - H_i)-
\chi(\cal H^{H_i}),
\end{equation}
where $\cal H^{H_i}$ is the restriction of $\cal H$ to a hyperplane $H_i$ and $\cal H - H_i$ is the arrangement with the hyperplane 
$H_i$ deleted.  
(We refer to Section \ref{sec:linalg} for complete definitions of deletion and restriction).
Similar deletion-restriction relations hold for other polynomials associated to hyperplane arrangements, including the Poincar\'e polynomial and 
the two-variable Tutte polynomial.  

In this paper we categorify these invariants, upgrading them from polynomials to homology theories.  Our constructions are modeled on the Odd Khovanov homology of  
Ozsv\'ath-Rassmussen-Szab\'o \cite{OzsvathRasmussenSzabo} which categorifies the Jones polynomial of links in the three-sphere.
In particular, the relation (\ref{eq:delres0}) becomes a long exact sequence in homology, as expected by analogy with the Skein relation for the Jones polynomial 
and the resulting long-exact sequence in (odd) Khovanov homology.

The first part of the paper considers homology theories for unsigned hyperplane arrangements.  These constructions are free from the restrictions imposed by 
isotopy invariance in the theory of link homologies, and as a result there is a lot of freedom in the definition of boundary maps between chain groups.  For 
example, in considering categorifications of the characteristic polynomial, we see that essentially the same chain groups can be made into a chain complex using 
two very different choices of boundary maps.  One choice of boundary maps gives a chain complex $\cal C_d(\cal H)$ which essentially generalizes to hyperplane 
arrangements earlier work \cite{HelmeRong} on categorification of the chromatic polynomial of a graph.  A second 
choice of boundary map, however, gives a completely different complex $\cal C_\partial(\cal H)$, which turns out to be compatible with a multiplication defined 
at the chain 
level; thus this second choice assigns a differential graded algebra to each central hyperplane arrangement.  Similar categorifications of the Poincar\'e and Tutte 
polynomials are defined in the body of the paper. A deletion-restriction triple $(\cal H,\cal H - H_i,\cal H^{H_i})$ gives 
rise to a short exact sequence of chain complexes in all of our homology theories, though the resulting long exact sequences of homology behave quite differently 
for the two different choices of boundary map.  

The second part of the paper considers signed hyperplane arrangements, that is,  arrangements with a sign assigned to each hyperplane.
The motivation for the consideration of signed arrangements comes from low-dimensional topology, as a planar projection of a framed link in the three-sphere defines in a natural way a signed hyperplane arrangement.  In low dimensional topology one considers link projections up to the equivalence relations generated by Reidemeister moves.  These moves, too, generalize to an equivalence relation on signed 
arrangements.  In Section \ref{sec:signed} we construct a bi-graded homology theory for signed arrangements which is invariant under generalized Reidemeister moves.  The Euler characteristic of this homology theory gives a  Reidemeister invariant ``Jones polynomial"  for signed arrangements.
The precise connection to topology is that, when restricted to arrangements coming from projections of framed links, our homology theory yields an invariant of framed links.  This invariant is quite closely related to the reduced Odd Khovanov homology of  \cite{OzsvathRasmussenSzabo}; since theirs is an invariant of ordinary (unframed) links, a writhe-dependent shift is needed to recover Odd Khovanov homology from our bi-graded theory (see Proposition \ref{prop:OddKH}).

All of the chain complexes we use are essentially straightforward modifications to hyperplane arrangements of existing constructions for 
links \cite{Khovanov,OzsvathRasmussenSzabo,Bloom}.  Nevertheless, hyperplane analogs of link homologies seem both sufficiently natural and 
sufficiently interesting as to warrant further investigation.  For example, the differential graded algebra structure on chain groups associated to 
unsigned arrangements does not 
(as far as we are aware) occur the polynomial categorifications of low-dimensional topology.  Moreover, hyperplane arrangements and their signed analogs 
admit a duality, known as Gale duality, which greatly generalizes duality of planar graphs.  The dual of a non-planar graph thus makes sense as a 
hyperplane arrangement.  The considerations in Section \ref{sec:signed} for signed hyperplane arrangements are perhaps the most interesting in the 
paper. In fact, the Gale duality statement for signed arrangements is cleaner than for ordinary arrangements, as there is an obvious isomorphism between the chain complexes assigned to a signed arrangement and its Gale dual.

For unsigned graphs, rather than hyperplane arrangements, constructions distinct but somewhat similar to ours appear in the earlier works\cite{HelmeRong,JassoRong}.
A completely different categorification of the Tutte polynomial for (unsigned) hyperplane arrangements was also given previously by Denham \cite{Denham}.  
There are some fundamental differences between our approach here and his; for example, the version of the Tutte polynomial which occurs as an Euler characteristic 
in our constructions has a rather different normalization than the one in \cite{Denham}.  It would interesting to relate the constructions in this paper to his work, 
and to several of the other basic algebraic structures in the combinatorics of hyperplane arrangements.  It should be possible to formulate many of the constructions of this paper for more general matroids; a homology theory which categorifies the Tutte polynomial of a matroid has also recently been investigated independently by A. Lowrance and M. Cohen.

Finally, we point out that it is an old question to find a recipe that, given an arrangement 
$\cal H$, produces a natural bi-graded vector space whose graded 
dimension is the Tutte polynomial of $\cal H$.  In light of the categorifications of low dimensional topology, it is also reasonable to modify this 
question and search instead for a chain complex whose graded Euler characteristic is the Tutte polynomial.  
Two of the homology theories of the current paper give a solution for hyperplane arrangements.  

\subsection{Acknowledgements}
The authors would like to thank the Institute for Advanced Study, where most of this research was carried out, for their hospitality
and support.  We would like to thank the anonymous referees for a number of useful comments and suggestions, and for noticing an error in a previous version of this paper.  A.L. would also thank Nick Proudfoot for several interesting conversations. 

\subsection{Linear algebra of vector arrangements}\label{sec:linalg}
Let $\cal V =\{V; \nu_1,...,\nu_n\}$ be a vector arrangement in a $k$-dimensional vector space $V$ over $\k$,
and define $\cal H = \{V; H_1,...,H_n\}$ to be the associated arrangement of 
hyperplanes in $V$.  The arrangement $\cal H$ is central, meaning that the intersection $\cap_{i} H_i \neq \emptyset$.
Let
\begin{equation}\label{eq:Wdef}
	W = \{(w_1,\hdots, w_n)\in \k^n \mid \sum_i w_i \nu_i = 0\}
\end{equation}
denote the space of linear dependencies in $\cal V$. This can be thought of as the orthogonal complement of span$\{\nu_i: i=1,...,n\}$ in $\k^n$. The inner product on $\k^n$ induces an inner product 
$\langle -,-\rangle$ on $W$, thus identifying $W$ with $W^*$.

To a subset $S\subset [n]$ there are three naturally associated vector spaces:
\begin{itemize}
\item $H_S = \cap_{i\in S} H_i$,
\item $V_S = span \{\nu_i\}_{i\in S}$, and
\item $W_S = \{w = (w_1,\hdots,w_n) \in W \mid w_r = 0 \text{ for } r\notin S\}$.
\end{itemize}
For $s \in S$ and $r\notin S$, there are natural inclusions
\begin{equation}\label{eq:inclusions}
	H_{S} \hookrightarrow H_{S-s}, \ \ V_S \hookrightarrow V_{S\cup r}, \ \ 
	W_S \hookrightarrow W_{S \cup r},
\end{equation}
and orthogonal projection maps
\begin{equation}\label{eq:projections}
 H_{S} \to H_{S\cup r}, \ \ V_S \to V_{S-s}, \ \ 
	W_S \to W_{S-s}.
\end{equation}

\begin{remark}\label{rm:defGale}\rm{
The spaces $H_S$ and $V_S$ are related by standard linear duality, in that $V_S$ is the 
space of vectors orthogonal to $H_S$ (equivalently, the space of linear functionals which vanish on it).  
The relationship between $H_S$ and $W_S$ is more subtle: the space of linear 
dependencies $W$ comes equipped with $n$ linear functionals, namely, the coordinate projections 
$\nu_i^\vee: w = (w_1,\hdots,w_n)\mapsto w_i$. Via the identification $W \cong W^*$, $\nu_i^\vee$ can be thought of
as the orthogonal projection of the standard basis vector $(0,...,1,...,0)$ onto $W$ (where 1 appears in the $i$-th coordinate). Thus 
$\cal V^\vee =\{W; \nu^\vee_1,...,\nu^\vee_n\}$ is another vector arrangement, known as the 
\emph{Gale dual} of $\cal V$.  Let $\cal H^\vee$ be the hyperplane arrangement associated to the vector arrangement $\cal V^\vee$; 
the defining hyperplanes of $\cal H^\vee$ are $H_i^\vee = \ker(\nu_i^\vee)$.  Then the space $W_S$ above is given by 
$W_S = H_{S^c}^\vee = \cap_{i\notin S} H_i^\vee$, where $S^c$ denotes the complement of $S$ (that is, $[n]\setminus S$).  
Thus the spaces $\{H_S\}_{S\subset [n]}$ and $\{W_S\}_{S\subset [n]}$ are 
exchanged by Gale duality. Note that if the vectors $\nu_i$ generate $V$, then the canonical inner products on $V$ and $W$ are related by 
$\langle \nu_i, \nu_j \rangle = -\langle \nu_i^\vee, \nu_j^\vee \rangle$. This follows from the fact that 
$\nu_i+\nu_i^\vee=x_i$ for all $i$, where the $x_i$ are the standard orthonormal basis vectors of $\k^n$.}
\end{remark}

The inclusions and projections (\ref{eq:inclusions}) and (\ref{eq:projections}) induce maps of exterior algebras. We will denote all the maps which increase the size of the subset $S$ by $d$'s, and all those which decrease subset size by $b$'s (``$b$'' is a backwards ``$d$'' ). 
Thus we have  
\begin{equation}\label{eq:diffmaps}
\extp^\bullet (H_{S}) \xrightarrow{d_{S,r}} \extp^\bullet (H_{S\cup r}), \ \ 
\extp^\bullet (V_S) \stackrel{d_{S,r}}{\hookrightarrow} \extp^\bullet (V_{S\cup r}), \ \ 
	\extp^\bullet (W_S) \stackrel{d_{S,r}}{\hookrightarrow} \extp^\bullet (W_{S \cup r}),
\end{equation}
and
\begin{equation}\label{eq:diffmapsback}
\extp^\bullet (H_{S}) \stackrel{b_{S,s}}{\hookrightarrow} \extp^\bullet (H_{S-s}), \ \ 
\extp^\bullet (V_S) \xrightarrow{b_{S,s}} \extp^\bullet (V_{S-s}), \ \ 
	\extp^\bullet (W_S) \xrightarrow{b_{S,s}} \extp^\bullet (W_{S-s}).
\end{equation}

However, the vectors from the original vector arrangement $\cal V$ can also be used to define maps between exterior algebras by wedging and contracting.  For $s \in S$, $r \notin S$, we define:
\begin{equation}\label{eq:wedgemaps}
\extp^\bullet H_{S} \xrightarrow{w_{S,s}} \extp^{\bullet+1} H_{S-s}, \ \ 
\extp^\bullet V_{S} \xrightarrow{w_{S,r}} \extp^{\bullet+1} V_{S \cup r}, \ \ 
	\extp^\bullet W_{S} \xrightarrow{w_{S,r}} \extp^{\bullet+1} W_{S \cup r}.
\end{equation}
\begin{equation}\label{eq:contractionmaps}
 \extp^\bullet H_{S} \xrightarrow{c_{S,r}} \extp^{\bullet-1} H_{S \cup r}, \ \ 
\extp^\bullet V_{S} \xrightarrow{c_{S,s}} \extp^{\bullet-1} V_{S}, \ \ 
	\extp^\bullet W_{S} \xrightarrow{c_{S,s}} \extp^{\bullet-1} W_{S-s}.
\end{equation}

The maps in (\ref{eq:wedgemaps}) act by wedging on the left by $\nu_s$, $\nu_r$, and $\nu_r^\vee$, respectively. 
In the first of these, note that $\nu_s$ is considered as an element of $H_{S-s}$ by orthogonal projection from $V$.  
Similarly, for the third map we consider $\nu_r^\vee$ 
as an element of $W_{S\cup r}$.

The maps in (\ref{eq:contractionmaps}) act by contraction with the vectors $\nu_r$, $\nu_s$ and $\nu_s^\vee$ respectively.
To define the first map, think of
$\nu_r$ as an element of $H_S$ by projecting it there, and note that for any element $h \in \extp^\bullet H_S$
the image $(\nu_r \perp h)$ lies in $\extp^{\bullet}H_{S\cup r}$, since $\nu_r$ is orthogonal to $H_{S \cup r}$.  Similarly, contraction with $\nu_s^\vee$ on $\extp^\bullet W_S$ has image in $\extp^{\bullet-1}W_{S-s}$ since $\nu_s^\vee$ is orthogonal to
$W_{S-s}$.   However, in the case of $V_S$, $\nu_s$ is not necessarily orthogonal to $V_{S-s}$, so the image does not lie in
$\extp^{\bullet-1}V_{S-s}$. Note also that in case of $H_S$ and $W_S$, the wedge and contraction maps are linear duals to each other
(via the identifications $H_S\cong H_S^*$ and $W_S\cong W_S^*$ induced by the inner product). This is not true for $V_S$: the linear dual of $w_{S,r}$
is $c_{S\cup r, r}$ composed with an orthogonal projection onto $V_{S-s}$.  

\begin{remark}\label{rmk:super}\rm{
All of the algebras we consider in this paper are naturally $\mathbb{Z}$-graded, and 
they will be considered as superalgebras for the $\mathbb{Z}_2$ grading induced from the $\mathbb{Z}$ grading.  Tensor products are always taken in the category of superalgebras.  Thus
$(a\otimes 1)(1 \otimes b)=(-1)^{\deg a \cdot \deg b}(1 \otimes b)(a \otimes 1)$, and
in this way there is an algebra isomorphism
$\extp^\bullet (V \oplus W) \cong \extp^\bullet V \otimes \extp^\bullet W$.}
\end{remark}

\begin{remark}\rm{We have taken $\k$ to be a field for convenience, but indeed almost all constructions in this paper may be carried out over the integers, 
or over an arbitrary commutative ring.  The one exception is in one part of Section \ref{sec:signed}, where we must work over a field (see Remark \ref{rem:field}).
}
\end{remark}

\subsection{Deletion and Restriction}\label{subseq:delres}
On the level of vector arrangements, a {\it subarrangment} of $\cal V$ is an arrangement in
the same ambient space consisting of a subset of the vectors in $\cal V$. For $\nu_i \in \cal V$,
the {\it deletion} of $\nu_i$ is the operation which results in the subarrangement with $\nu_i$
removed, denoted $\cal V - \nu_i$. Given $\nu_i \in \cal V$, the {\it restriction} $\cal V^{\nu_i}$
is an arrangement in the orthogonal complement $\nu_i^\perp \subseteq V$, consisting of the vectors
$\{P(\nu_j): j \neq i\}$, where $P$ is the orthogonal projection to $\nu_i^\perp$. 

The corresponding notions for hyperplane arrangements follow from the above.
A {\it subarrangement} of $\cal H$ is an arrangement consisting of a subset of
hyperplanes in $\cal H$, in the same ambient vector space. The arrangement obtained
by deleting the hyperplane $H_i$ from $\cal H$ is denoted $\cal H-H_i$.

Given a hyperplane $H_i \in \cal H$, the {\it restriction} of $\cal H$ to $H_i$ is the arrangement \linebreak
$\cal H^{H_i}=\{H_i; H_1\cap H_i,...,H_{i-1}\cap H_i,
H_{i+1}\cap H_i,...,H_k\cap H_i\}$. 

Deletion and restriction are Gale dual notions: $(\cal H-H_i)^\vee=(\cal H^\vee)^{H_i^\vee}$ 
and $(\cal H^{H_i})^\vee=\cal H^\vee-H_i^\vee$, and similarly for vector arrangements.

\newpage
\subsection{Polynomials associated to Hyperplane arrangements}\label{sec:polys}
\subsubsection{The Characteristic Polynomial}\label{subseq:charpoly}
As before, for
$S\subseteq [n]$, let $H_S=\bigcap_{s\in S}H_s$.  
The characteristic polynomial (see \cite{TeraoOrlik}) of the central hyperplane arrangement $\cal H$ 
is defined as
\begin{equation}\label{eq:statesum}
\chi(\cal H, q) = \sum_{S \subseteq [n]} (-1)^{|S|} (1+q)^{\dim H_S}.
\end{equation}
This is a slightly non-standard normalization, and the definition in \cite{TeraoOrlik} would be given by the substitution $t=1+q$. 

There is a closely related, though distinct, polynomial that also occurs as an Euler characteristic in categorification: 
\begin{equation}\label{eq:statesum2}
\bar{\chi}(\cal H, q)=\sum_{S\subseteq [n]}q^{|S|}(1-q)^{\dim H_S}
\end{equation}

Informally we will refer to both of the above polynomials as characteristic polynomials.

\subsubsection{The Poincar\'e Polynomial}\label{subseq:poinpoly}
The Poincar\'e polynomial of a hyperplane arrangement $\pi(\cal H, t)$ contains the same information as the characteristic
polynomial, as they are related by a change of variables \cite{TeraoOrlik}.  
A convenient state sum
definition of the Poincar\'e polynomial (in a slightly unusual normalization) is
\begin{equation}\label{eq:statesumpi}
\pi(\cal H, q)= \sum_{S\subseteq [n]} (-1)^{|S|} (1+q)^{\dim V_S}
\end{equation}

\subsubsection{The Tutte Polynomial}\label{seubseq:tuttepoly}
The Tutte polynomial of $\cal H$ is usually defined as
$$T(\cal H; x,y)= \sum_{S \subseteq [n]} (x-1)^{\dim H_S - \dim H_{[n]}}(y-1)^{\dim W_S}.$$
The version we will categorify is the analogue of the version of the Tutte polynomial of graphs used in \cite{HelmeRong}, given by the state sum formula
\begin{equation}\label{eq:statesumtutte}
\hat{T}(\cal H; x,y)= \sum_{S \subseteq [n]}(-1)^{|S|} (1+x)^{\dim H_S}(1+y)^{\dim W_S}.
\end{equation}
The relationship between these two polynomials is given by the following formula, where $k$ denotes the dimension of $V$:
$$\hat{T}(\cal H; x,y)=(-1)^k(-x-1)^{dim H_{[n]}}T(\cal H;-x,-y).$$

\subsubsection{Relations from deletion-restriction}\label{subseq:delres}
All of the polynomials defined above satisfy {\it deletion-restriction 
formulas}. If $H_l\in\cal H$ is a given hyperplane in the arrangement $\cal H$, 
$\cal H - H_l$
is the subarrangement produced by deleting $H_l$ from $\cal H$, and $\cal H^{H_l}$
is the restriction to $H_l$, then 
\begin{equation}\label{eq:delres}
 \chi(\cal H, q)=\chi(\cal H - H_l,q)-
\chi(\cal H^{H_l},q),
\quad \text{and} \quad
\bar{\chi}(\cal H, q)=\chi(\cal H - H_l,q)+
q\chi(\cal H^{H_l},q). 
\end{equation}

Similar relations hold for the Poincar\'e and Tutte polynomials if $H_l$ is non-degenerate, i.e., if $\nu_l \neq 0$:
\begin{equation}\label{eq:delrespoin}
 \pi(\cal H, q)=\pi(\cal H - H_l)-(1+q)\pi(\cal H^{H_l}), 
\end{equation}
and
\begin{equation}\label{eq:delrestutte}
 \hat{T}(\cal H; x, y)=\hat{T}(\cal H - H_l; x,y)-\hat{T}(\cal H^{H_l}; x,y).
\end{equation}

\section{Categorifications for unsigned arrangements}
In this section we will describe several homology theories which categorify the
characteristic, Poincar\'e, and Tutte polynomials. The different constructions arise from 
the freedom to choose between the $H_S$, $V_S$ and $W_S$ spaces (and their tensor products)
to build chain groups, and between the natural inclusion/projection maps versus the wedge/contraction 
maps to construct differentials. We develop categorifications of the two characteristic polynomials, using the spaces $H_S$, in this section.  The homology theories categorifying the Poincar\'e and Tutte polynomials are very similar, and we will only state the results and highlight where the proofs differ from those for the characteristic polynomials.

\subsection{Hypercubes associated to a hyperplane arrangement}\label{subsec:cubes}
In the vein of \cite{Khovanov} and \cite{HelmeRong},
we use the state sum formula (\ref{eq:statesum}) to construct a chain complex,
the graded Euler characteristic of which is the characteristic polynomial
by design.
The first step
is to arrange the terms of the formula on the vertices of a cube,
in this case the vertices correspond to subsets $S \subseteq [n]$. 
The space $\extp^\bullet H_S$ is placed at the vertex corresponding to $S$. 

We illustrate the cube on the example of the {\it braid arrangement} in $\bb R^3$. 
This arrangement consists of three hyperplanes
defined by the vector arrangement $\nu_1=(1,-1,0)$, $\nu_2=(0,1,-1)$, and
$\nu_3=(-1,0,1)$.
By placing the spaces $\extp^\bullet H_S$ at vertices, and connecting
them by an edge if the subsets $S$ differ only by one element, 
we obtain the following 3-dimensional cube:

\begin{center}
 \input cube.pstex_t
\end{center}
Note that the vertices of the cube are organized into columns according to the size of the
subsets $S$.

In the next section we will discuss the maps associated to the cube edges which make up the
differentials; here we only define the chain groups, which are obtained by ``flattening'' the cube along the ``columns''.  Thus we get chain groups
$$C^i=\bigoplus_{S \subseteq [n], \, |S|=i} \extp^\bullet H_S.$$

Thus the chain complex $\cal C$ is given by 
\begin{equation}\label{eq:cubedef}
\cal C= \bigoplus_{S \subseteq [n]} C_S, \text{ where } C_S=\extp^\bullet H_S.
\end{equation}
This vector space has a natural bi-grading given by $\deg \extp^j H_S=(|S|, j)$.
Note that the Euler characteristic with respect to the first grading component is 
$$\chi_q(\cal C)=\sum_{S \subseteq [n]} (-1)^{|S|}(1+q)^{\dim H_S}=\chi(\cal H, q).$$
Thus, if we impose differentials for which the homological degree of $C_S$ is $|S|$, 
as we will do in our first construction, the graded Euler characteristic of the complex will be the characteristic polynomial.  
In the second construction, discussed in Section \ref{subsubsec:partial}, we will need to shift the grading, and the resulting graded Euler characteristic will yield the second characteristic polynomial
$\bar{\chi}(\cal H,q)$.

There are also natural cubes involving the spaces $V_S$,  
$$\cal C^P=\bigoplus_{S \subseteq [n]}\extp^\bullet V_S,$$ leading to a categorification
of the Poincar\'e polynomial.  

To categorify the Tutte polynomial, we will use the tensor product of the spaces $H_S$ and $W_S$. 
$$\cal C^T=\bigoplus_{S \subseteq [n]}\extp^\bullet H_S \otimes \extp^\bullet W_S.$$
This vector space is triply-graded, with 
$\deg \extp^i H_S \otimes \extp^j W_S= (|S|, i, j)$. As with the chain groups used to categorify the characteristic polynomial, 
one choice of differential will be homogeneous for this grading convention, while another choice will require us to make minor 
adjustments to the definition of the gradings.

\subsection{Boundary maps}\label{subsec:boundary}
The first class of categorifications uses differentials which arise from
the natural inclusion and orthogonal projection maps
explained in Section \ref{sec:linalg}. In fact these inclusion and projection maps can be used to define chain 
complexes in even more general settings; for
 example, instead of exterior algebras, we could use instead symmetric algebras, or in principle any other exact 
functor from vector spaces to graded vector spaces.  Moreover, these boundary maps are rather straightforward odd generalization from graphs to hyperplane arrangements of the differentials defined in \cite{HelmeRong, JassoRong}.

In contrast, the boundary maps defined in Section \ref{subsubsec:partial} use the wedge maps as differentials.  
These maps use in a fundamental way the structure in the exterior algebra, and the resulting chain complexes have a dg-algebra structure and simpler theorems for deletion/restriction.

\subsubsection{Boundary maps arising from natural inclusion or orthogonal projection}\label{subsubsec:d}
We describe the construction for the characteristic polynomial
in detail. Recall that for each edge of the cube corresponding
to a subset $S \subseteq [n]$ and an element $r \notin S$
there are maps induced by the orthogonal projections
$$d_{S,r}: \extp^\bullet H_S \to \extp^\bullet H_{S \cup r}.$$
These are of degree $(1,0)$ with respect to the natural bi-grading
defined in Section \ref{subsec:cubes}.

To define differentials, we want to take the direct sums of the maps. 
However, as in Khovanov homology, we first need to introduce signs to make 
the (a priori commutative)
cube anti-commutative, this is needed for the square of the differential to be zero. 
To achieve this, we set
$$\varepsilon_{S,r}=\begin{cases}
                        -1 & \text{ if } |\{s \in S, s<r\}|= \text{odd} \\
			 1 & \text{ otherwise.}
                       \end{cases}
$$
The differentials are the sums with appropriate signs of the 
edge maps $d_{S,r}$ going from the algebras in column $i$ to the ones in column $(i+1)$:
$$d^i=\bigoplus_{|S|=i,\; r \notin S} \varepsilon_{S,r} d_{S,r}.$$
Let us illustrate this on the braid arrangement example:

\begin{center}
 \input fullexample.pstex_t
\end{center}

\begin{definition}\label{def:main}\rm{
We denote the resulting 
chain complex by $C_d^\bullet(\cal H, \k)$, 
and the homology by $H_d^\bullet(\cal H, \k)$, and call it (odd)
{\it characteristic homology}.
}
\end{definition}

\begin{proposition}\label{prop:CharCat}
The homology $H_d^\bullet(\cal H, \k)$ has graded Euler characteristic equal to the characteristic polynomial:
$$\chi_q(H_d^\bullet(\cal H, \k))=\chi(\cal H, q).$$
\end{proposition}

\begin{proof}
The graded Euler characteristic of the chain complex is the characteristic polynomial
by design, as noted in Section \ref{subsec:cubes}. As the chain groups are finite dimensional
and the differential is degree zero with respect to the second grading, the graded Euler characteristic of the homology is the same.
\end{proof}

Note that the sign assignment for the differentials made use of the ordering
of the hyperplanes. The following lemma states that the end result is, up to isomorphism,
order-independent.

\begin{lemma}\label{thm:permute}
For any permutation $\sigma \in S_n$ 
and arrangement $\cal H= \{V; H_1,...,H_n\}$, let \linebreak
$\cal H_{\sigma}:=\{V; H_{\sigma(1)},...,H_{\sigma(n)}\}$ denote the permuted arrangement.
Then
$$H^\bullet_d(\cal H_{\sigma})\cong H^\bullet_d(\cal H).$$
\end{lemma}

\begin{proof}
Since $S_n$ is generated by transpositions, it is enough to prove the theorem for \linebreak
$\sigma=(i,i+1)$, $i\in\{1,...,n-1\}$.

We prove that the chain complexes $\cal C(\cal H)$ and $\cal C(\cal H_\sigma)$ are isomorphic.
Acting by $\sigma$ does not change the chain groups, it just permutes the
direct summands of a fixed chain group. However, some of the signs for the differentials differ in   $\cal C(\cal H)$ and $\cal C(\cal H_\sigma)$.  Specifically, $\epsilon_{S,r}$ changes if and only if $r=i$ and $(i+1) \in S$ or $r=i+1$ and $i \in S$.

An isomorphism of the chain complexes 
$\Phi: \cal C(\cal H) \to \cal C(\cal H_{\sigma})$ is given by letting
$\Phi$ be multiplication by $(-1)$ on the components $\extp^\bullet H_S$ where
$\{i,i+1\} \subseteq S$, and letting $\Phi$ act by the identity on all other summands.
It is simple combinatorics to check that this map commutes with the differentials,
hence it gives rise to a chain isomorphism.
\end{proof}

\begin{example}\label{ex:baby}{\rm
One could in principle construct a chain complex from the spaces $H_S$ directly, rather than first taking the exterior algebra.  
The resulting complex is less interesting, however, as this simple example will illustrate.
Consider the hyperplane arrangement in $\bb R^3$ consisting of two planes defined by vectors $\nu_1=(1,-1,0)$ and 
$\nu_2=(0,1,-1)$. The cube in this case is a square. The proposed ``simple''
complex,
\begin{center}
\input babyexample.pstex_t 
\end{center}
is acyclic.  However, by taking exterior algebras before flattening the cube, 
we get non-zero homology in homological degree 0,
linearly spanned by the two elements 
$\{x_1x_2x_3, x_1x_3-x_1x_2-x_2x_3\}$.

Note that the graded Euler characteristic of this is in fact $q^3+q^2$, in agreement with the characteristic polynomial.  \qed
}
\end{example}

\smallskip 

We can assign differentials to the $V_S$ spaces the same way:
the complex $\cal C^P_d=\bigoplus_{S \subseteq [n]}\extp^\bullet V_S$ with the natural bi-grading and 
differentials
$$d=\bigoplus_{S \subseteq [n]; r \notin S} \varepsilon_{S,r} d_{S,r}$$
gives rise to homology groups $H^{P\bullet}_d(\cal H)$.

Similarly, consider 
$\cal C^T_d:=\bigoplus_{S \subseteq [n]}\extp^\bullet H_S \otimes \extp^\bullet W_S$
with the triple grading $\deg \extp^i H_S \otimes \extp^j W_S= (|S|, i, j)$ and
differentials
$$d=\bigoplus_{S \subseteq [n]; r \notin S} \varepsilon_{S,r} d_{S,r}\otimes d_{S,r},$$
and call the resulting homology $H^{T\bullet}_d(\cal H)$.

The proof of the theorems regarding the characteristic polynomial (Proposition \ref{prop:CharCat} and Lemma \ref{thm:permute}) can be repeated
word by word to prove the following:
\begin{proposition}
 The cohomologies $H_d^{P\bullet}(\cal H)$ and $H^{T\bullet}_d(\cal H)$ categorify the Poincar\'e polynomial
and the Tutte polynomial, respectively:
$$\chi_q(H_d^{P\bullet}(\cal H))=\pi(\cal H, q), \text{ and } \chi_q(H_d^{T\bullet}(\cal H))=\hat{T}(\cal H, x,y).$$
Furthermore, permuting the vectors in the vector arrangement induces isomorphisms of the chain groups.
\end{proposition}

\subsubsection{Wedge maps}\label{subsubsec:partial}
We now define a second type of differentials using the same underlying chain groups but shifting the grading. 
This construction has several advantages over the previous one, most notably these chain complexes admit a differential 
graded algebra structure, so the cohomologies are themselves algebras. 

In this section we only construct characteristic and Tutte complexes -- that is, the complexes which use the spaces $H_S$ and $W_S$. The
Poincar\'e complex of the previous section used the spaces $V_S$. As we have seen in Section \ref{sec:linalg}, the wedge maps between the
$V_S$ spaces go in the opposite direction (increasing the size of the set $S$). Later in this section we will construct dg-algebra using the spaces $H_S$ and $W_S$; we note here that it is also possible to define a dg algebra structure using the spaces $V_S$, though we have omitted the 
details of that construction below.

For the characteristic homology, the differential $\partial: \cal C \to \cal C$ is defined using the wedge maps
$w_{S,s}: \extp^\bullet H_S \to \extp^{\bullet+1}H_{S-s}$ explained in 
Section \ref{sec:linalg}, Equation (\ref{eq:wedgemaps}). We set
$$\partial:=\bigoplus_{S \subseteq [n], \, s \in S} w_{S,s}.$$ 
From now on we will denote the maps $w_{S,s}$ by the name $\partial_{S,s}$ as well.
To distinguish between the chain complexes with different differentials when needed,
we will write $\cal C_d$ and $\cal C_\partial$.

Note that the hypercube with vertices $\extp^\bullet H_S$ and edge maps
$\partial_{S,s}$ is anti-commutative by definition, hence $\partial^2=0$.

\begin{remark}{\rm
 The anti-commutativity of the hypercube is not completely obvious, due to the different orthogonal
projections used in the definition of the wedge maps (see Section \ref{sec:linalg}), though it is straightforward to verify
that it is in fact anti-commutative. 
Alternatively, an essentially equivalent way to rephrase the above construction in a way that makes anti commutativity obvious would be
to view the vectors
$\nu_i$ as elements of the linear dual $V^*$, and take for chain groups the dual spaces $H^*_S$.  Then orthogonal projection is 
replaced by restriction of linear maps, making the anti-commutativity of the hypercube more transparent.  On the other hand, in such 
a formulation it is no longer obvious that the wedge maps are well defined -- though again this is straightforward to check -- 
which is why we have chosen the presentation above.}
\end{remark}

Observe that $\deg \partial=(-1,1)$ with respect to the natural bi-grading of the previous section. Thus in this section we will redefine the bi-grading by setting $\deg \extp^i H_S=(|S|+i, i)$. With respect to this grading $\deg \partial=(0,1)$.

Let $H^\bullet_{\partial}$ denote the resulting homology. (Note that the two gradings switched roles: 
now the second degree is the homological degree.) In this construction we never used that the hyperplanes were ordered, so $H^\bullet_\partial$ is order-independent for free.  Note also that there were no sign choices required in the definition of the differential on the complex.

\begin{lemma}
The graded Euler characteristic of $H^\bullet_{\partial}$ is
$$\chi_q(H^\bullet_{\partial})(\cal H)=\bar{\chi}(\cal H, q),$$
where $\bar{\chi}$ is the version of the characteristic polynomial defined in (\ref{eq:statesum2}).
\end{lemma}

\begin{proof}
Immediate from the definition.
\end{proof}

As for the Tutte polynomial, there are several possible differentials to chose from. One condition which is convenient to impose is that the cube should be anti-commutative naturally, without the order-dependent sign assignments. One such differential is as follows.

We define the differentials on $\cal C^T$ by
\begin{equation}\label{eq:Tuttepartial}
\partial^T=\bigoplus_{S \subseteq [n], s \in S} \partial^T_{S,s}, \text{ where } \partial^T_{S,s}= w_{S,s}\otimes b_{S,s}.
\end{equation}
With respect to the natural grading, $\deg \partial^T=(-1,1,0)$. We set the new grading convention to
be $\deg \extp^i H_S \otimes \extp^j W_S=(|S|+i, i, j)$. In this grading $\deg \partial^T =(0,1,0)$.
We denote the homology by $H^{T\bullet}_\partial$.

\begin{proposition}
The graded Euler characteristic of $H^{T\bullet}_\partial$ is
 $$\chi_q(H^{T\bullet}_\partial(\cal H))= (-x)^k \hat{T}(\cal H;-\frac{1}{x},-1-x-xy).$$ 
\end{proposition}

\begin{proof}
A straightforward check.
\end{proof}

\subsection{Differential graded algebra structure}\label{subsec:dgalg}
One advantage of the wedge differentials of the previous subsection 
is that the resulting homology groups admit a compatible multiplication.

For the characteristic homology, this multiplication is defined at the chain level as a map
$$
	m: C_S\otimes C_T \rightarrow C_{S\cup T}.
$$
We set $m$ to be $0$ when $S \cap T\neq \emptyset$; for $S \cap T = \emptyset$ and $h \in C_S, \, h'\in C_T$, we set 
$$
	 m(h \otimes h') =h\wedge h'.
$$
Here the wedging takes place inside 
$\extp^\bullet(H_{S\cup T})$, which is well-defined after first using orthogonal projection to send
both $h$ and $h'$ to $\extp^\bullet(H_{S\cup T})$.  If $h$ and $h'$ are homogeneous elements of respective bi-degrees $(|S|+i,i)$ and $(|T|+j,j)$,
then $$\deg m(h \otimes h')= \deg (h \wedge h')= (|S|+|T|+i+j,i+j),$$ so the multiplication
respects both gradings. Note that multiplication is associative due to the fact that a composition of orthogonal projections to progressively 
smaller subspaces is an orthogonal projection.

\begin{proposition}
$(\cal C, \partial, m)$ is a differential graded algebra,
and hence $H_{\partial}^\bullet(\cal H)$ is a graded algebra.
\end{proposition}
\begin{proof}
We need to show that the multiplication is compatible with the differential:
 $$\partial(m(h\otimes h'))=m((\partial h)\otimes h') + (-1)^{j} m(h \otimes (\partial h')),$$
where $j=\deg_2(h)$ is the second degree of $h$ (i.e., its exterior algebra degree).
A short computation shows that both sides are equal to 
$\sum_{r \in S\cup T} \nu_r \wedge h \wedge h' \in \bigoplus_{r \in S \cup T} C_{S\cup T -r}$.
\end{proof}

For the Tutte chain groups, we define the multiplication on $\cal C^T_\partial(\cal H)$
in a similar way:
$$
m^T: (\extp^{i_1}H_S \otimes \extp^{j_1} W_S)\otimes
(\extp^{i_2}H_T \otimes \extp^{j_2} W_T) \to (\extp^{i_1+i_2}H_{S\cup T} \otimes \extp^{j_1+j_2} W_{S\cup T}),
$$
\begin{equation}\label{eq:Tuttemult}
m^T((h_1\otimes w_1)\otimes(h_2\otimes w_2))=(h_1\wedge h_2) \otimes (w_1 \wedge w_2) \quad \text{ if } \quad S\cap T=\emptyset,
\end{equation}
and set the multiplication to be zero when $S \cap T\neq \emptyset$. 

\begin{proposition}
The multiplication $m^T$ is compatible with all three gradings and makes $\cal C^T$ into a triply-graded dg-algebra.  As a result, $H^{T\bullet}_\partial$ is a triply-graded algebra.
\end{proposition}

\begin{proof}
 What needs to be verified is that for 
$$h_1\otimes w_1 \in \extp^{i_1}H_S \otimes \extp^{j_1}W_S \quad \text{and} \quad
h_2\otimes w_2 \in \extp^{i_2}H_S \otimes \extp^{j_2}W_S,$$
$$m^T(\partial^T(h_1\otimes w_1)\otimes (h_2 \otimes w_2))+(-1)^{i_1} m^T((h_1\otimes w_1)\otimes \partial^T(h_2\otimes w_2))=
\partial^T(m(h_1\otimes w_1 \otimes h_2 \otimes w_2)).$$
This is a straightforward calculation.
\end{proof}

\subsection{Properties}\label{subsec:props}

\subsubsection{Relations from deletion and restriction}\label{subsec:delres}
The following theorem is a categorification of the deletion-restriction formula (\ref{eq:delres}):

\begin{theorem}\label{thm:les}
There is a short exact sequence of chain complexes of the form
$$
 0 \to C_d^{i-1,j}(\cal H^{H_l}) \stackrel{\iota}{\to} C_d^{i,j}(\cal H)
\stackrel{\pi}{\to} C_d^{i,j}(\cal H - H_l)\to 0.
$$
This induces long exact sequence for $H^\bullet_d$:
\begin{equation}\label{eq:les}
0\to...\to H^{i-1}_d(\cal H^{H_l}) \to H^i_d(\cal H) 
\to H^i_d(\cal H - H_l) \to H^i_d(\cal H^{H_l}) \to...
\end{equation} 
\end{theorem}

\begin{proof}
The proof is along the same lines as the proofs of the corresponding theorems in \cite{HelmeRong} 
and \cite{JassoRong}, and we recall the basic points here.

We want to define chain maps $\iota$ and $\pi$ satisfying 
\begin{equation}\label{eq:sesd}
0 \to \bigoplus_{S \subseteq [n]-l} \extp^\bullet(H_S\cap H_l) \stackrel{\iota}{\to}
\bigoplus_{T \subseteq [n]} \extp^\bullet H_T \stackrel{\pi}{\to}
\bigoplus_{U \subseteq [n]-l} \extp^\bullet H_U \to 0.
\end{equation}
Note that 
$$\bigoplus_{T \subseteq [n]} \extp^\bullet H_T=
\bigoplus_{S \subseteq [n]-l} \extp^\bullet H_{S\cup l} \oplus \bigoplus_{U \subseteq [n]-l} \extp^\bullet H_U,$$
and $H_S \cap H_l=H_{S\cup l}$. The essential idea is to set $\iota$ to be the natural inclusion and 
$\pi$ the natural projection map with respect to this decomposition. However, this is only correct up to sign:
$\pi$ commutes with the differential $d$, but the signs $\varepsilon_{S,r}$
cause a commutativity issue with $\iota$. In order to fix this,
we replace the inclusion $\iota$ by the map $\iota'$ by setting
$\iota'=\bigoplus_{S\subseteq [n]-l} \iota'_S$,
where $\iota'_S$ is the natural inclusion of the component if the number of elements 
$\{s \in S, s>l\}$ is even, and multiplication by $(-1)$ on the component if this
number is odd.
\end{proof}

For the second choice of differentials, there is a similar short exact sequence of chain complexes inducing a 
long exact sequence on cohomology groups.  
Let $[1]$ denote a shift by $(1,0)$ of the bi-grading (remember that the homological grading is 
the second one in this case).

\begin{theorem}\label{thm:lesdel}
There is a short exact sequence of chain complexes
\begin{equation}\label{eq:sesdel}
 0 \to \cal C_\partial(\cal H - H_l) \stackrel{\iota}{\to} \cal C_\partial(\cal H)
\stackrel{\pi}{\to} \cal C_\partial(\cal H^{H_l})[1]\to 0,
\end{equation}
giving rise to a long exact sequence 
\begin{equation}\label{eq:lesdel}
... \to H^i_\partial(\cal H - H_l) \to H^i_\partial(\cal H) \to H^{i}_{\partial}(\cal H^{H_l})[1] \to H^{i+1}_\partial(\cal H - H_l) \to ...
\end{equation}
\end{theorem}

\begin{proof}
It is easy to check that Equation (\ref{eq:sesdel}) with the natural inclusion and projection maps is indeed 
a short exact sequence
of chain complexes (not of dg-algebras) for $(\cal C, \partial)$.  
This implies that there is a long exact sequence 
for the homology $H^\bullet_\partial$. 
\end{proof}

\begin{remark}
 \rm{Note that in the short exact sequence of chain complexes, $\iota$ is a map
of dg-algebras, but $\pi$ isn't. Correspondingly in the long exact sequence the
maps induced by $\iota$ are ``algebra maps'' in the sense that they fit into commutative
squares with the corresponding multiplication maps. On the other hand, the maps induced
by $\pi$ and the ones arising from the snake lemma have no such properties.}
\end{remark}

The proofs above relied crucially on the fact that for a subset $S \subseteq [n]-l$,
the space at hypercube vertex $S$ in the chain complex of $\cal H^{H_l}$ (denote this by 
$H_S^{H_l}$) is $H_S\cap H_l$.  This vector space can be identified with $H_{S\cup l}$, 
which participates in the chain complex of $\cal H$ at vertex $S\cup l$.  

Before we state the deletion-restriction theorems for the Poincar\'e and Tutte cohomologies, 
let us determine what the analogous hypercube relationships are for the $W$ and $V$ spaces.

For $S \subseteq [n]-l$, let $W_S^{H_l}$ denote the ``$W_S$-space'' of the vector arrangement associated to $\cal H^{H_l}$. 
That is, $W_S^{H_l}$ is the space of linear dependencies between $\{P(\nu_s): s \in S\}$, where $P$ stands for the 
orthogonal projection onto $H_l$. Note that $w_{s_1} P(\nu_{s_1})+...+w_{s_p} P(\nu_{s_p})=0$
if and only if $w_{s_1} \nu_{s_1}+...+w_{s_p} \nu_{s_p}=w_l \nu_l$. Let $\phi$ be the map sending
the vector $\underline{w}$ with non-zero coordinates $w_{s_1},...,w_{s_p}$ to $\tilde{\underline{w}}+(0,...,0,-w_l,0,...,0)$,
where $\tilde{\underline{w}}$ is the same as $\underline{w}$ but with a 0 inserted in the $l$-th coordinate.
Then $\phi$ is a canonical isomorphism $W_S^{H_l} \cong W_{S\cup l}$. 

For the $V_S$ spaces, the analogous statement is somewhat different, since for $l \notin S$, \linebreak
$\dim \extp^\bullet V_{S\cup l} =2 \dim \extp^\bullet V_S^{H_l}$, as long as $\nu_l \neq 0$.  There is a natural inclusion 
$\iota_1: \extp^\bullet V_S^{H_l} \hookrightarrow \extp^\bullet V_{S\cup l}$
as vector spaces, as follows. Recall that $V_S^{H_l}$ is spanned by $\{P(\nu_s): s \in S\}$, where $P$ is the orthogonal projection to $H_l$.
For $v\in \extp^\bullet V_S^{H_l}$ given by $v=P(\nu_{s_1})\wedge...\wedge P(\nu_{s_i})$, define $\iota_1(v)=\nu_l \wedge \nu_{s_1}\wedge...\wedge \nu_{s_i}$. 
It is a simple exercise to check that
$\iota_1$ is well-defined, injective,  and that $\deg \iota_1=(1,1)$ (working with the grading convention $\deg(\extp^i V_S)=(|S|,i)$). 
Furthermore, note that $P(\nu_s)$ is an element of $V_{S\cup l}$ for each $s \in S$.
Define $\iota_2:\extp^\bullet V_S^{H_l} \hookrightarrow \extp^\bullet V_{S\cup l}$ to be the identity on each $P(\nu_S)$ and extend multiplicatively to the
exterior algebra to get a different injection, with $\deg \iota_2=(1,0)$. It is easy to see that the images of $\iota_1$ and $\iota_2$ only intersect at 0.  Hence 
$\iota_1 \oplus \iota_2: \extp^\bullet V_S^{H_l}\oplus \extp^\bullet V_S^{H_l} \rightarrow
\extp^\bullet V_{S\cup l} $ is an isomorphism of vector spaces, as long as $\nu_l \neq 0$.

\begin{theorem}\label{thm:lespoinc}
When $\nu_l \neq 0$, there is a short exact sequence of chain complexes 
$$
 0 \to \cal C^P_d(\cal H^{H_l})\{1\}[1]\oplus \cal C^P_d(\cal H^{H_l})[1] \stackrel{\iota_1\oplus \iota_2}{\longrightarrow} \cal C^{P}_d(\cal H)
\stackrel{\pi}{\to} \cal C^{P}_d(\cal H - H_l)\to 0,
$$
where $[1]$ stands for shifting the first degree up by 1, and $\{1\}$ denotes shifting the second (exterior algebra) degree up by 1.  
This induces long exact sequence for $H^{P\bullet}_d$ of the form
\begin{equation}\label{eq:lespoinc}
0\to...\to H^{P,i-1}_d(\cal H^{H_l})\{1\} \oplus H^{P,i-1}_d(\cal H^{H_l}) \to H^{P,i}_d(\cal H) 
\to H^{P,i}_d(\cal H - H_l) \to H^{P,i}_d(\cal H^{H_l}) \to... 
\end{equation} 
where $H^{P,i-1}_d(\cal H^{H_l})\{1\}$ denotes a shift of the second (exterior algebra) degree up by 1.  
\end{theorem}

\begin{proof}
Identical to the proof of Theorem \ref{thm:les}, but using $\iota_1 \oplus \iota_2$ in place of $\iota$.
Again, $\pi$ commutes with $d_{S,r}$, but on account of the sign assignments 
$\varepsilon_{S,r}$, $\iota_1\oplus \iota_2$ needs to be adjusted in the same way as in 
the proof of Theorem \ref{thm:les}.
\end{proof}

\begin{theorem}\label{thm:lestutte}
When $\nu_l \neq 0$, there is a short exact sequence of chain complexes 
$$
 0 \to \cal C^T_d(\cal H^{H_l})[1] \stackrel{\iota}{\to} \cal C^{T}_d(\cal H)
\stackrel{\pi}{\to} \cal C^{T}_d(\cal H - H_l)\to 0,
$$
where $[1]$ denotes a shift of the first degree; inducing a
long exact sequence for $H^{T\bullet}_d$ of the form
\begin{equation}\label{eq:lestutte}
0\to...\to H^{T,i-1}_d(\cal H^{H_l}) \to H^{T,i}_d(\cal H) 
\to H^{T,i}_d(\cal H - H_l) \to H^{T,i}_d(\cal H^{H_l}) \to...
\end{equation} 
\end{theorem}

\begin{proof}
The proof of Theorem \ref{thm:les} applies verbatim.
\end{proof}

\begin{theorem}\label{thm:lesdeltutte}
There is a short exact sequence of (tri-graded) chain complexes
$$
 0 \to \cal C^T_\partial(\cal H - H_l) \stackrel{\iota}{\to} \cal C^T_\partial(\cal H)
\stackrel{\pi}{\to} \cal C^T_\partial(\cal H^{H_l})[1]\to 0,
$$
where $[1]$ denotes a degree 
shift by $(1,0,0)$ in the tri-grading. 
Thus there is an induced long exact sequence of cohomology groups:
\begin{equation}\label{eq:lesdeltutte}
...\to H^{T,i}_\partial(\cal H - H_l) \to H^{T,i}_\partial(\cal H) \to H^{T,i}_{\partial}(\cal H^{H_l})[1]\to H^{T,i+1}_\partial(\cal H - H_l) \to...
\end{equation}
\end{theorem}

\begin{proof}
Same as the proof of Theorem \ref{thm:lesdel}.
\end{proof}

\subsubsection{Factorization property}
For two vector arrangements $\cal V=\{V^k;\nu_1,...,\nu_n\}$ and $\cal V'=\{V'^l; \nu'_1,...,\nu'_m\}$,
the {\it product arrangement} is defined to be the vector arrangement
$$\cal V \times \cal V'=\{V\times V'; \nu_1\times 0,...,\nu_m\times 0,0\times\nu'_1,...,0\times\nu'_m\}.$$

For the associated hyperplane arrangements $\cal H=\{V^k; H_1,...,H_n\}$
and $\cal H'=\{V'^l; H_1',...,H_m'\}$, the
product arrangement is the hyperplane arrangement given by
$$\cal H \times \cal H'=\{V \times V' ;\;
H_1 \times V',...,H_n \times V', V \times H_1',...,
V \times H_m'\}.$$
\begin{theorem}\label{thm:kunneth} The characteristic homology of the product arrangement is
the tensor product of the characteristic homologies of the factors:
$$H_d^\bullet(\cal H \times \cal H')\cong H_d^\bullet(\cal H) \otimes H_d^\bullet(\cal H'),
\hskip.2cm H_\partial^\bullet(\cal H \times \cal H')\cong H_\partial^\bullet(\cal H) \otimes H_\partial^\bullet(\cal H').
$$
\end{theorem}

\begin{proof}
The product arrangement has $n+m$ hyperplanes,
and subsets of $[n+m]$ are in one-to-one correspondence with pairs of subsets $S \subseteq [n]$
and $T \subseteq [m]$. 
Note that $$\bigcap_{s \in S} H_s \times V' = H_S \times V', \; \text{ and }\;
(H_S \times V') \cap (V \times H_T) = H_S \times H_T,$$ so the
components of the chain complex for the product arrangement are products
of those for $\cal H$ and $\cal H'$. In other words, the component in the
product arrangement corresponding to the pair of subsets $S,T$ is
$H_{S,T}= H_S \times H_T$. Then

$$\cal C ({\cal H} \times {\cal H'})= \bigoplus_{S\in [n],\, T \in [m]} \extp^\bullet H_{S,T}
\cong \bigoplus_{S\in [n],\, T \in [m]} \extp^\bullet H_S \otimes \extp^\bullet H_T
= \cal C(\cal H) \otimes \cal C(\cal H'),$$
as bi-graded vector spaces, with both grading conventions.

Regarding the differentials, for $h \in \extp^j H_S$ and $h' \in \extp^{j'} H_T$, we need to
check that $$d_{\cal H \times \cal H'} (h \otimes h') =d_{\cal H}(h) \otimes h' +(-1)^{|S|} h \otimes d_{\cal H'}(h'),$$
and that $$\partial_{\cal H \times \cal H'} (h \otimes h') =\partial_{\cal H}(h) \otimes h' +(-1)^{j} h \otimes \partial_{\cal H'}(h').$$
This is verified by a direct computation in both cases.

It remains to show that the algebra structure on $H^\bullet_\partial(\cal H \times \cal H')$
is the (super) tensor product of the algebra structures on the factors: for $a,b \in \cal C(\cal H)$ and 
$a',b' \in \cal C(\cal H')$ we want 
that 
$$m_{\cal H \times \cal H'}(a \otimes a', b \otimes b')= (-1)^{\deg_2 a \deg_2 b} m(a,b) \otimes m(a',b'),$$
where $\deg_2$ denotes the second (exterior algebra) degree.
This is straightforward from the definition of multiplication.
\end{proof}

The essential ingredient of the proof above was that $H_{S,T}=H_S\times H_T$. It is straightforward from the
definitions that this holds true for the $V_{S,T}$ and $W_{S,T}$ spaces associated to the product arrangement 
as well: $V_{S,T}=V_S \times V_T$ and $W_{S,T}=W_S \times W_T$. Hence the proof can be repeated without change to produce similar
``K\"unneth Theorems'' for the Poincar\'e and Tutte cohomologies.

\begin{proposition}
For the Poincar\'e and Tutte cohomologies of the product arrangement, we have 
$$H^{P\bullet}(\cal H \times \cal H')\cong H^{P\bullet}(\cal H) \otimes H^{P\bullet}(\cal H'),$$
$$H^{T\bullet}_d(\cal H \times \cal H')\cong H^{T\bullet}_d(\cal H) \otimes H^{T\bullet}_d(\cal H'),$$
$$H^{T\bullet}_\partial(\cal H \times \cal H')\cong H^{T\bullet}_\partial(\cal H) \otimes H^{T\bullet}_\partial(\cal H'),$$
where the first two are isomorphisms of bi- and tri-graded vector spaces, and the last is an isomorphism of tri-graded algebras. \qed 
\end{proposition}

\subsubsection{Gale duality and Tutte homology}\label{subsubsec:GaleAndTutte}
It is a well-known fact that Gale duality switches the variables of the Tutte polynomial.
In other words, if $\cal H^\vee$ is the Gale dual arrangement to $\cal H$, then
$\hat{T}(\cal H^\vee,x,y)=\hat{T}(\cal H, y,x)$. This is a direct consequence of
Remark \ref{rm:defGale}. In this section we consider the relationship between the Tutte 
homology of an arrangement and its Gale dual.

Recall that for Tutte homology we had some freedom in choosing the differentials; in fact
our definition of boundary maps for the Tutte complex $\cal C^T=\bigoplus_{S \subseteq [n]}\extp^\bullet H_S \otimes \extp^\bullet W_S$ was one of four equally natural choices:
$$\partial_1\!= \! \! \! \bigoplus_{S\subseteq [n], s \in S} \! \! \! w_{S,s}\otimes b_{S,s}, \quad
\partial_2\!=\! \! \! \bigoplus_{S\subseteq [n], s \in S} \! \! \! b_{S,s}\otimes c_{S,s}, \quad
\partial_3\!=\! \! \! \bigoplus_{S\subseteq [n], r \notin S} \! \! \! d_{S,r}\otimes w_{S,r}, \quad
\partial_4\!=\! \! \! \bigoplus_{S\subseteq [n], r \notin S} \! \! \! c_{S,r}\otimes d_{S,r}.
$$
(We used the map $\partial_1$ in our earlier definitions).  With respect to the natural triple grading, these maps are of the following degrees:
$$\deg \partial_1=(-1,1,0), \quad \deg \partial_2=(-1,0,-1), \quad \deg \partial_3=(1,0,1), \quad 
\text{and} \quad \deg \partial_4=(1,-1,0).$$
To construct four chain complexes, we set $\deg \extp^i H_S \otimes \extp^j W_S$ to be
$$(|S|+i,i,j), \quad (|S|-j, i, j), \quad (|S|-j, i, j), \quad \text{and} \quad (|S|+i, i,j),$$
respectively. In these conventions, the differentials are of degrees
$$(0,1,0), \quad (0,0,-1), \quad (0,0,1), \quad \text{and} \quad (0,-1,0). $$

Of these differentials, $\partial_1$ and $\partial_4$ are linear duals of each other, as are 
$\partial_2$ and $\partial_3$. On the other hand, $\partial_1$ is related to $\partial_3$ by Gale 
duality, and similarly $\partial_2$ is Gale dual to $\partial_4$. Let us demonstrate what we mean by this on $\partial_1$ and $\partial_3$.

As explained in Remark \ref{rm:defGale}, if $\cal H^\vee$ is the Gale dual arrangement to $\cal H$, then $H_S^\vee=W_{S^c}$ and $W_S^\vee=H_{S^c}$, where $S^c$ is the complement of the set $S$
in $[n]$.

So we have an isomorphism (as vector spaces) $\varphi : \cal C^T(\cal H) \to \cal C^T(\cal H^\vee)$,
where $\varphi$ sends the component $\extp^\bullet H_S \otimes \extp^\bullet W_S$ isomorphically (by switching the 
tensor factors) to $\extp^\bullet W_S \otimes \extp^\bullet H_S$. The latter is the component of  $\cal C^T(\cal H^\vee)$
corresponding to the subset $S^c \subseteq [n]$.

The isomorphism $\varphi$ intertwines the differential $\partial_1$ on $C^T(\cal H)$ with the differential 
$\partial_3$ on $C^T(\cal H^\vee)$, i.e. for $x \in \extp^\bullet H_S \otimes \extp^\bullet W_S$,
$\varphi(\partial_1(x))=\partial_3(\varphi(x))$. So $\varphi$ is an isomorphism 
of chain complexes $\cal C^T_{\partial_1}(\cal H) \to \cal C^T_{\partial_3}(\cal H^\vee)$, except for the fact that it does not respect the grading: it sends the component of degree $(|S|+i,i,j)$ in $\cal C^T_{\partial_1}(\cal H)$ to the component of degree $(n-(|S|+i), j,i)$ in $\cal C^T_{\partial_3}(\cal H^\vee)$.

There is a dg-algebra structure on all four chain complexes. For $\cal C^T_{\partial_1}$ this was defined in Section \ref{subsec:dgalg}. For $\cal C^T_{\partial_2}$ the same definition of multiplication works. For $\cal C^T_{\partial_3}$ 
and $\cal C^T_{\partial_4}$ multiplication is defined in the following way. For subsets $S$ and $T$ of $[n]$,
multiplication is a map 
$$m: (\extp^\bullet H_S \otimes \extp^\bullet W_S)\otimes (\extp^\bullet H_T \otimes \extp^\bullet W_T)
\to \extp^\bullet H_{S\cap T} \otimes \extp^\bullet W_{S \cap T}.$$
If $S \cup T \neq [n]$ then $m$ is defined to be 0, otherwise $m$ is ``wedging'':
$m(h \otimes w \otimes h' \otimes w')=(h \wedge h') \otimes (w \wedge w')$, where we use
the natural inclusions to interpret $h$ and $h'$ as elements of $\extp^\bullet H_{S\cap T}$,
and orthogonal projections to interpret $w$ and $w'$ as elements of $\extp^\bullet W_{S \cap T}$.
We leave it to the reader to check that $m$ is compatible with the differentials $\partial_3$
and $\partial_4$.

As $(S \cup T)^c=S^c \cap T^c$, and $S \cap T=\emptyset$ if and only if $S^c \cup T^c=[n]$,
$\varphi$ is not only compatible with the differentials but also an algebra homomorphism (by 
an easy verification). Hence $\varphi$ descends to an algebra isomorphism on the homology.

When we consider signed arrangements in Section \ref{sec:signed}, we will see that the Gale duality statements for Tutte homology 
become even cleaner, though there will no longer be a dg-algebra structure on chain groups.

\subsection{Examples}

\subsubsection{Graphical Arrangements}\label{graphconst}
 
Important examples of central hyperplane arrangements are known  as graphical arrangements, associated to a finite graph.  
We compute a few examples of homology groups  $H_d^\bullet(\cal H)$ and $H_\partial^\bullet(\cal H)$ associated to such arrangements 
in this section.  The homology theories $H_d^\bullet$ and $H^{T,\bullet}_d$, when restricted to graphical arrangements, may be 
seen as odd versions of the graph homologies considered in \cite{HelmeRong} and \cite{JassoRong}, respectively. In contrast, 
the restrictions of the homology theories $H_\partial^\bullet(\cal H)$ and $H_\partial^T$ to graphs seem quite different 
from ones considered by those authors. 

To a finite directed graph $G$ with vertex set $V(G)=\{v_1,...,v_k\}$ and ordered edge set $E(G)$, we associate a
vector arrangement $\cal V(G)=\{\k^k; \{\nu_e\}_{e \in E(G)}\}$, consisting of $|E(G)|$ vectors, as follows.  If an edge $e \in E(G)$
starts at vertex $v_i$ and ends at $v_j$, the corresponding vector in the arrangement $\cal V(G)$
is $\nu_e=x_i-x_j$.
Let us denote the hyperplane arrangement arising from $\cal V(G)$ by $\cal H(G)$. 
The hyperplane arrangements that arise from graphs via this construction
are called graphical arrangements, and the characteristic polynomial of hyperplane arrangements specializes to the chromatic polynomial of 
graphs when restricted to graphical arrangements.

For a subset $S\subseteq E(G)$, denote by $G_S$ the subgraph which contains all the vertices of $G$ but only the edges
in $S$. Note that then by construction the dimension of $H_S$ equals the number of connected components of $G_S$, 
and the dimension of $W_S$ equals the rank of the cycle space of $G_S$
(if $G_S$ is planar then this is equal to the number of bounded faces). 

It is sometimes convenient to consider the hyperplane arrangement associated to a graph as living 
in a slightly smaller ambient vector space.
Note that the line given by the equation $x_1=...=x_k$ is always included in
the intersection $H_{[n]}$ of all the hyperplanes in $\cal H(G)$. 
So we may consider a graphical vector arrangement or hyperplane arrangement
modulo this subspace. We denote these arrangements by 
$\bar{\cal V}(G)$ and $ \bar{\cal H}(G)$; the hyperplane arrangement then consists of 
$|E(G)|$ hyperplanes living in
$\k^{k-1}$. In this case the dimension of the space $H_S$ is one less than the number of connected components of $G_S$.

One advantage of viewing the hyperplane arrangement associated to a graph as living in this smaller space is that planar graph duality corresponds 
to Gale duality: for a connected planar graph $G$ and $G^*$ its planar dual,  $\bar{\cal H}(G^*)=\bar{\cal H}(G)^\vee$.
Similarly, for any graph $G$ the Tutte polynomial of this associated arrangement equals the Tutte polynomial of the graph: $T(\bar{\cal H}(G); x,y)=T(G;x,y)$.

\begin{lemma}
 For the empty arrangement $\cal H_0^k$ of no hyperplanes in $V=\k^k$,
$$H^0_d(\cal H_0^k)\cong \extp^\bullet \k^k, \text{ and } H^i_d(\cal H_0^k) = 0 \text{ for } i\neq 0.$$
On the other hand, $$H^\bullet_\partial(\cal H_0^k)=\extp^\bullet \k^k$$
as bi-graded algebras.
\end{lemma}

\begin{proof}
This is straightforward from the definitions. 
\end{proof}

The following statement is the analogue of the computation for trees done in \cite{HelmeRong}.
The proof is essentially the same, so we only provide a sketch.  In the statement below, hyperplanes $\{H_i,\hdots,H_n\}$ are said to be 
linearly independent if their associated normal vectors 
$\{\nu_1,...,\nu_n\}$ are linearly independent.

\begin{proposition}
For a hyperplane arrangement with a maximal number of linearly independent hyperplanes 
$\cal H_n=\{\k^{n}; H_1,...,H_n\}$, 
$H^0_d(\cal H_n)=\k\{n\},$ and all other homology groups are zero.
\end{proposition}

\begin{proof}
We use induction. The case of $n=0$ is trivial.
Note that $$\cal H_n^{H_n}=\cal H_{n-1} \; \text{ and } \; \cal H_n -H_n= \cal H_{n-1} \times \k,$$
where $\k$ stands for the empty arrangement in a 1-dimensional $\k$-vector space. Hence
so $$H^i_d(\cal H_n-H_n)=H^i_d(\cal H_{n-1})\otimes \extp^\bullet \k=H^i_d(\cal H_{n-1}) \oplus H^i_d(\cal H_{n-1})\{1\}.$$
For each $i \geq 0$ we have
$$...\to H^i_d(\cal H_n) \to H^i_d(\cal H_{n-1}) \oplus H^i_d(\cal H_{n-1})\{1\} \stackrel{\gamma}{\to} H^i_d(\cal H_{n-1}) \to...$$
where $\gamma$ is the transition map arising from the snake lemma. Working through the snake lemma one can see that for 
$h \in H^i_d(\cal H_{n-1})$, $\gamma(h,0)=h$. Hence $\gamma$ is surjective and the long exact sequence falls apart to split short
exact sequences, implying the result. 
\end{proof}

For $H_\partial$ the total dimension is the same, only the grading differs:
\begin{proposition}
$H^0_\partial(\cal H)=\k,$ and all other homology groups are zero.
\end{proposition}

\begin{proof}
The proof is very similar to the one above, let us point out the differences only. Due to the different grading convention,
$H_\partial^0(\k)=\k$ as well as $H_\partial^1(\k)=\k$. Thus, $H^i_\partial(\cal H_n -H_n) \cong H^i_\partial(\cal H_{n-1})\oplus H^{i-1}_\partial(\cal H_{n-1})$. 
By Theorem \ref{thm:lesdel} we then have
$$...\to H^i_\partial(\cal H_{n-1})\oplus H^{i-1}_\partial(\cal H_{n-1}) \to H^i_\partial(\cal H_n) \to H^i_\partial(\cal H_{n-1})[1] 
\stackrel{\beta}{\to} H^{i+1}_\partial(\cal H_{n-1})\oplus H^{i}_\partial(\cal H_{n-1}) \to...$$
In this case, the transition map of the snake lemma turns out to be injective (it is ``wedging with $\nu_n$''), so the long exact sequence
falls apart to short exact sequences, and the statement follows by induction on $n$.
\end{proof}

Note that from these results one can compute the characteristic homology of all hyperplane arrangements with no dependencies amongst the hyperplanes, 
as these are products of some $\cal H_n$ with an empty arrangement.

\begin{proposition}
If the arrangement $\cal H$ contains a degenerate hyperplane $H$ (i.e., $H=V$), then $H^\bullet_d(\cal H)=0$.
\end{proposition}
 
\begin{proof}
 The proof is the same as that of the corresponding theorem for loop edges in \cite{HelmeRong}.
If $H$ is degenerate, then $\cal H^H=\cal H-H$, and for each $i$ the transition map 
$\gamma: H^i_d(\cal H -H) \to H^i_d(\cal H^H)$ is an isomorphism, which implies that
then $H^\bullet_d(\cal H)=0$.
\end{proof}

In contrast, $H^\bullet_\partial$ is not necessarily zero for arrangements that contain degenerate hyperplanes.
For example, the total dimension of $H^\bullet_\partial$ for $\k$ with a single degenerate hyperplane is 4, as
the only differential in the complex is zero.

\section{Signed hyperplane arrangements and Odd Khovanov homology}\label{sec:signed}
In this section we consider signed arrangements: vector arrangements together with a sign associated to each vector.  
Just like vector arrangements generalize graphs in a sense, signed vector arrangements 
generalize signed graphs.
They arise naturally from low dimensional topology: the checkerboard coloring of a planar projection of a link
gives rise to a signed planar graph, which in turn has an associated signed arrangement.
A topologist considers planar link projections up to the equivalence relation generated 
by Reidemeister moves, which characterize isotopies of the link in $\bbR^3$.  
This equivalence relation generalizes naturally to signed graphs \cite{BollobasRiordan}, and signed vector arrangements, and via 
this generalization Reidemeister invariance questions may be posed for polynomials or chain complexes associated to signed arrangements.

We define a version of Tutte homology for signed hyperplane arrangements, which we call framed odd Khovanov homology,
and prove hyperplane Reidemeister 
invariance for these homology groups. When restricted from signed hyperplane arrangements to planar link projections, this Tutte homology 
is a framed reduced version of the odd Khovanov homology  of Ozsv\'ath-Rasmussen-Szab\'o \cite{OzsvathRasmussenSzabo}.  
We should mention that in the case of a planar link projection, our chain complex is not identical to the chain complex 
of \cite{OzsvathRasmussenSzabo}; instead it is closely related to the chain complex defined by Bloom in \cite{Bloom}, wherein he provides a chain 
homotopy between his complex and that of Ozsv\'ath-Rasmussen-Szab\'o.

This section is organized as follows: first, without discussing links, we define Reidemeister moves, the framed Jones polynomial, and framed odd Khovanov homology
for signed arrangements and prove their main properties (long exact sequences, K\"unneth theorem and most importantly Reidemeister invariance). In 
Section \ref{subsec:links} we discuss the procedure of associating a signed vector arrangement to a planar link diagram and 
the relationship of the framed Jones polynomial
and odd Khovanov homology of arrangements to the well-known Jones polynomial and odd Khovanov homology of links.

\subsection{Signed arrangements, Reidemeister moves and the Jones polynomial}\label{subsec:SignedArrnments}
A signed vector arrangement is a vector arrangement $\cal V=\{V; \nu_1,...,\nu_n\}$ together with an 
assignment of a sign ($+$ or $-$) to each vector $\nu_i$.  The hyperplane arrangement associated to a signed vector 
arrangement is referred to as a signed hyperplane arrangement, since the sign associated to each vector can be thought 
of as a sign attached to the associated hyperplane.  Since all constructions in this section will be carried out for 
signed arrangements, we use the same notation as we did for unsigned arrangements in previous sections; thus in the 
notation $\cal V=\{V; \nu_1,...,\nu_n\}$, it is understood that each $\nu_i$ is a vector 
together with a sign.  Similarly, we denote by $\cal H=\{V; H_1,...,H_n\}$ the signed hyperplane arrangement associated to $\cal V$.

Let $\cal V=\{V; \nu_1,...,\nu_n\}$ be a signed vector arrangement.  The sign assignment partitions the set $[n]$ into 
subsets $[n]=[n]_+ \sqcup [n]_-$, where $[n]_+$ is the
set of vectors assigned $+$ and and $[n]_-$ is the set vectors assigned $-$. For
a subset $S \subseteq [n]$ we write $S_+=S \cap [n]_+$ and $S_-=S \cap [n]_-$. 

Deletion and restriction of signed arrangements is defined just as for ordinary arrangements.  We extend Gale duality from arrangements 
to signed arrangements as follows.    If $\cal V=\{V; \nu_1,...,\nu_n\}$ is a signed arrangement, the Gale dual 
$\cal V^\vee = \{W; \nu^\vee_1,...,\nu^\vee_n\}$ is, as an unsigned arrangement, the Gale dual of the unsigned arrangement $\cal V$.
The sign assignment of $\cal V^\vee $ is given by declaring that the sign associated to $\nu^\vee_i$ is the opposite of the 
sign assigned to $\nu_i$.

Now we define Reidemeister moves for signed vector arrangements. The motivation is link theory:
planar projections of isotopic (framed) links may be obtained from one another by a sequence of Reidemeister moves.  
Reidemeister moves have been generalized to signed graphs by Bollob\'as and Riordan, for a detailed 
discussion we refer the reader to \cite{BollobasRiordan} and references therein.  
Here we simply state the moves for signed arrangements; readers motivated by link homology may prefer to read Section \ref{subsec:links} 
up to Proposition \ref{prop:RmovesLinks} first, and readers familiar with the Bollob\'as--Riordan article will see that the signed arrangement Reidemeister moves
are straightforward generalizations of signed graph moves.

Reidemeister moves come in Gale dual pairs, denoted $Ri$ and $Ri^\vee$ for $i=1,2,3$. We will re-state this more precisely after listing 
the Reidemeister moves (see Proposition \ref{prop:DualPairs}); here we mention for motivation that the corresponding graph phenomenon
is that signed graph Reidemeister moves are in planar dual pairs, which in turn correspond to opposite checkerboard shadings of a link
diagram, as discussed in Section \ref{subsec:links}. Finally, the $w$ preceding $R1$ in our notation stands for ``weak'', as it is a weaker
version of the $R1$ moves in \cite{BollobasRiordan}, and preserves the isotopy class {\em and total framing} of links, as explained in
Section \ref{subsec:links}. 

\begin{itemize}
\item $wR1$: If $\nu_l=\nu_m=0$ for some $l$ and $m$ of opposite signs, then $\cal V \leftrightarrow \cal V - \{\nu_l ,\nu_m\}$.
\item $wR1^\vee$: If $\nu_l^\vee=\nu_m^\vee=0$ (i.e., $\nu_l$ and $\nu_m$ are independent of $\cal V - \nu_l$ and $\cal V -\nu_m$, respectively), for some $l$ and $m$
of opposite signs,
then $\cal V \leftrightarrow \cal V^{\{\nu_l,\nu_m\}}$.
\item $R2$: If $\nu_l=\alpha \nu_m$, for some non-zero $\alpha \in \k$ and $l \neq m$ are of opposite signs, 
then $\cal V \leftrightarrow \cal V-\{\nu_l,\nu_m\}$.
\item $R2^\vee$: If $\nu_l^\vee=\alpha \nu_m^\vee$, with non-zero $\alpha \in \k$, for some $l \neq m$ of opposite signs, then $\cal V \leftrightarrow \cal V^{\nu_l,\nu_m}$. 
\item $R3$: Suppose there are three distinct vectors $\nu_l, \nu_m$ and $\nu_p$ in $\cal V$ with $l,m \in [n]_+$ and $p \in [n]_-$, 
and a linear dependence $\nu_l^\vee=\alpha_m \nu_m^\vee+ \alpha_p \nu_p^\vee$
with non-zero coefficients $\alpha_m,\alpha_p$. Then $\cal V \leftrightarrow \cal V'$, where $\cal V'$ is the arrangement obtained from $\cal V^{\nu_l}$ 
by adding an extra vector $\nu_l'=\alpha_p \nu_m-\alpha_m\nu_p$. The signs in $\cal V'$ are the same as those in $\cal V$
except that the sign of $l$ changes from positive to negative, i.e., $[n]_-' =[n]_-\cup l$.
\item $R3^\vee$: The same statement as $R3$, but with opposite sign assignments.

\end{itemize}

Note that the $wR1$ and $wR1^\vee$ moves are special cases of $R2$ and $R2^\vee$, respectively, where
the vectors are zero. We separarate the definitions as above in order to have a clear correspondence between
the Reidemeister moves above and Reidemeister moves of framed links (see Section \ref{subsec:links}), and because the case of 
$wR1$ and $wR1^\vee$ needs separate treatment in the proof of Theorem \ref{thm:Reidemeister}.

\begin{proposition}\label{prop:DualPairs}
The pairs of arrangement Reidemeister moves $wR1$ and $wR1^\vee$, $R2$ and $R2^\vee$, and $R3$ and $R3^\vee$ 
are Gale duals in the sense that the following square commutes:
 $$\xymatrix{
  \cal V \ar[r]^\vee \ar[d]^{Ri} & \cal V^\vee \ar[d]^{Ri^\vee} \\
  Ri(\cal V) \ar[r]^-\vee & (Ri(\cal V))^\vee=Ri^\vee(\cal V^\vee)  
}$$
\end{proposition}

\begin{proof}
 This is a routine check, let us demonstrate it for Reidemeister 2 only. Let $\cal V=\{V; \nu_1,...,\nu_n\}$, and assume for simplicity that
 $\nu_1=\alpha \nu_2$, with $\alpha \neq 0$, $1\in [n]_+$ and $2 \in [n]_-$. 
 
 The Gale dual of $\cal V$ is $\cal V^\vee=\{W, \nu_1^\vee,...,\nu_n^\vee\}$. Since $(\nu_i^\vee)^\vee=\nu_i$, $\nu_1^\vee$ and $\nu_2^\vee$ 
 satisfy the condition for $R2^\vee$. Applying $R2^\vee$ results in $(\cal V^\vee)^{\{\nu_1^\vee,\nu_2^\vee\}}$. 
 
 Performing $R2$ on $\cal V$ results in $\cal V-\{\nu_1,\nu_2\}$. Since deletion and restriction are Gale dual notions (as mentioned in Section \ref{subsec:delres}),
 the Gale dual of $\cal V-\{\nu_1,\nu_2\}$ is $(\cal V^\vee)^{\{\nu_1^\vee,\nu_2^\vee\}}$, as needed.
\end{proof}

Given a signed arrangement $\cal V=\{V; \nu_1,...,\nu_n\}$, we define the (normalized) framed Jones polynomial of $\cal V$ to be
a polynomial in $\mathbb Z[i,q^{1/2},q^{-1/2}]$, where $i^2=-1$. $J(\calV)$ is given by the state sum formula 
\begin{equation}\label{eq:Jones}
J(\cal V)= \sum_{S\subseteq [n]} (-1)^{|S|-n/2}q^{|S|-n/2}(q+q^{-1})^{\dim H_{S_+ \cup S_-^c}+\dim W_{S_+ \cup S_-^c}}.
\end{equation}
Here $S_-^c$ denotes the complement of $S_-$ in $[n]_-$. We will denote $S_+\cup S_-^c$ by $\tilde{S}$ in the future. 

In Section \ref{subsec:links}, Proposition \ref{prop:JonesAndJones} will discuss how $J(\cal V)$ is related to the normalized Jones polynomial (\ref{eq:UsualJones}) of links. 
Readers motivated by link theory may wish to read Section \ref{subsec:links} up to the end of the proof of Proposition \ref{prop:JonesAndJones} before proceeding.

\begin{proposition}
The framed Jones polynomial of a hyperplane arrangement is a hyperplane Reidemeister invariant.
\end{proposition}
Since this proposition follows by taking Euler characteristics in Theorem \ref{thm:Reidemeister}, we won't give 
an independent proof of it here.

\subsection{Odd Khovanov homology for signed arrangements}
To a signed vector arrangement $\cal V$ we associate a $(\mathbb Z \cdot \frac{1}{2})$-bi-graded chain complex (i.e., each degree is an integer or half integer). 
We denote this complex by $\cal T(\cal V)$ and define it as follows. 

As usual, we build a cube of chain groups by associating to each subset $S \subseteq [n]$ a bi-graded vector space 
$$T_S=\extp^\bullet ( H_{\tilde{S}}\oplus W_{\tilde{S}}) 
\cong \extp^\bullet H_{\tilde{S}}\otimes \extp^\bullet W_{\tilde{S}}.$$ 
Note that if $s\in[n]_+$ then $s\in \tilde{S}$ if and only if $s \in S$, while if $r \in [n]_-$ then $r \in \tilde{S}$ iff $r \notin S$.

We define the bi-grading on this complex by setting 
$$\deg \extp^i H_{\tilde{S}} \otimes \extp^j W_{\tilde{S}}=(|S|-n/2, |S|+\dim H_{\tilde{S}}+\dim W_{\tilde{S}}-2(i+j)-n/2).$$
These global shifts will
be needed for Reidemeister invariance. Note that the complex is either entirely contained in integer bi-degrees or in half integer bi-degrees, 
depending on the parity of $n$.
We will refer to the first grading as the homological grading and to the second as the $q$-grading. 

For $r\notin S$ the map $\delta_{S,r}: T_S \to T_{S \cup r}$ is defined as

\begin{equation}\label{eq:delta}
\delta_{S,r}=\begin{cases}
                        d_{\tilde{S},r}\otimes d_{\tilde{S},r} & \text{ if } r\in[n]_+ \text{ and } \nu_r\notin V_{\tilde{S}} \text{ (Type 1)}\\
			d_{\tilde{S},r}\otimes w_{\tilde{S},r} & \text{ if } r\in[n]_+ \text{ and } \nu_r \in V_{\tilde{S}} \text{ (Type 2)}\\
			b_{\tilde{S},r}\otimes b_{\tilde{S},r} & \text{ if } r\in[n]_- \text{ and } \nu_r\in V_{\tilde{S}-r} \text{ (Type 3)}\\
			w_{\tilde{S},r}\otimes b_{\tilde{S},r} & \text{ if } r\in[n]_- \text{ and } \nu_r\notin V_{\tilde{S}-r} \text{ (Type 4)}.
                       \end{cases}
\end{equation}
Here $d_{\tilde{S},r},w_{\tilde{S},r},b_{\tilde{S},r}$ are the the same maps from Section \ref{sec:linalg} that have been used 
in the construction of the chain groups associated to unsigned hyperplane arrangements.

The differential of the chain complex is defined as $\delta=\bigoplus_{S \in [n], r \notin S} \epsilon_{S,r} \delta_{S,r}$ for 
some appropriate choices of scalars $\epsilon_{S,r}$ specified below. Note that with respect to the bi-grading defined above,
$\deg \delta= (1,0)$. 

\begin{lemma}\label{lem:almostcomm}
In the cube of chain groups $T_S$ and edge maps $\delta_{S,r}$, every square face commutes up to a scalar, i.e. for any $r,t \notin S$,
$$\delta_{S\cup r, t } \circ \delta_{S,r} = \alpha_{S,r,t} \delta_{S \cup t, r} \circ \delta_{S,t},$$
for some scalars $\alpha_{S,r,t} \in \k$. 
\end{lemma}

\begin{proof}
 The proof amounts to checking several cases depending on which type of differential each side of the square belongs to.
The first major case is when $r,t \in [n]_+$. This breaks down into four sub-cases, as follows:
\begin{itemize}
 \item If $\nu_r \in V_{\tilde{S}}$ and $\nu_t \in V_{\tilde{S}}$, then all four edge maps of the square are of Type 2,
so the square is anti-commutative.
\item If $\nu_r \notin V_{\tilde{S}\cup t}$ and $\nu_t \notin V_{\tilde{S}\cup r}$, then all edge maps are of Type 1, 
and the square commutes.
\item If $ \nu_r \in V_{\tilde{S}}$ and $\nu_t \notin V_{\tilde{S}\cup r}$, then both $r$-edges are of Type 2, while
both $t$-edges are of Type 1, so the square commutes.
\item The most interesting case is when $\nu_r \notin V_{\tilde{S}}$ but $\nu_r \in V_{\tilde{S}\cup t}$, which implies that 
$\nu_t \notin V_{\tilde{S}}$ but $\nu_t \in V_{\tilde{S}\cup r}$. In this case there are two edges of Type 1 and two of Type 2,
namely $\delta_{S,r}=d \otimes d$, $\delta_{S,t}=d \otimes d$, $\delta_{S \cup t, r}=d \otimes w$, and $\delta_{S\cup r, t }=d \otimes w$.

Take $x \otimes y \in T_S=H_{\tilde S} \otimes W_{\tilde S}$. The two sides of the equality we need to check are:
$$\delta_{S\cup r, t } \circ \delta_{S,r}(x \otimes y)=x \otimes \nu^\vee_t \wedge y 
\in H_{\tilde{S}\cup r\cup t}\otimes W_{\tilde{S}\cup r\cup t},$$ and
$$\delta_{S \cup t, r} \circ \delta_{S,t}(x \otimes y)=x \otimes \nu^\vee_r \wedge y 
\in H_{\tilde{S}\cup r\cup t}\otimes W_{\tilde{S}\cup r\cup t}.$$
Recall that both $\nu^\vee_r$ and $\nu^\vee_t$ are interpreted in $W_{\tilde{S}\cup r\cup t}$ via orthogonal
projections, and due to the condition that  $\nu_r \notin V_{\tilde{S}}$ but $\nu_r \in V_{\tilde{S}\cup t}$,
it follows that their projections onto $W_{\tilde{S}\cup r\cup t}$ only differ by a scalar $\alpha_{S,r,t}$.
\end{itemize}

The second major case to consider is when $r \in [n]_+$ and $t \in [n]_-$. This breaks into five sub-cases, four of
which are trivial (the square either commutes or anti-commutes). 
\begin{itemize}
\item
The one interesting sub-case is when 
$\nu_r\notin V_{\tilde{S}-t}$ but $\nu_r \in V_{\tilde{S}}$, which implies that 
$\nu_t \notin V_{\tilde{S}-t}$ but $\nu_t \in V_{\tilde{S}-t\cup r}$. Again, consider 
$x \otimes y \in T_S=H_{\tilde S} \otimes W_{\tilde S}$. The two sides of the equality turn out to be
$$\delta_{S\cup r, t } \circ \delta_{S,r}(x \otimes y)=x \otimes \nu^\vee_r \wedge y 
\in H_{\tilde{S}\cup r\cup t}\otimes W_{\tilde{S}\cup r- t},$$ and
$$\delta_{S \cup t, r} \circ \delta_{S,t}(x \otimes y)=\nu_t \wedge x \otimes y 
\in H_{\tilde{S}\cup r\cup t}\otimes W_{\tilde{S}\cup r- t}.$$
The condition $\nu_r\notin V_{\tilde{S}-t}$, $\nu_r \in V_{\tilde{S}}$ implies that when projected
orthogonally onto $H_{\tilde{S}\cup r\cup t}$ and $W_{\tilde{S}\cup r- t}$, respectively, both
$\nu_t$ and $\nu_r^\vee$ map to zero, so the square commutes.
\end{itemize}

The third major case, when $r,t \in [n]_-$, is similar to the first, so we leave it to the reader to check.
\end{proof}

\begin{lemma}\label{lem:assign}
 There is a non-zero scalar assignment $\epsilon_{S,r} \in \k^\times$ to each edge to make the above cube anti-commutative,
and hence the flattened cube is a chain complex. Furthermore, choosing a different scalar assignment
does not change the homology of the complex.
\end{lemma}

\begin{proof}
 The proof is the same homological argument as the proof of the corresponding statements (Lemmas 1.2 and 2.2)
in \cite{OzsvathRasmussenSzabo}, so we only give a brief outline here. Consider the hypercube $\cal T(\cal V)$ as a cell complex, oriented by
the ordering of the vectors in $\cal V$.
Define a 2-cochain $c \in C^2(\cal T(\cal V), \k^\times)$ by associating to each face the negative of the scalar which obstructs the 
commutativity of the square, described in Lemma \ref{lem:almostcomm} and its proof. More precisely if $c_{S,r,t}$ denotes the scalar 
assigned to the 2-face of the cube corresponding to the subset $S$ and $r<t$ both in $S^c$, then 
$\delta_{S\cup t, r}\circ\delta_{S,t}=-c_{S,r,t}\delta_{S\cup r,t}\circ\delta_{S,r}$. When both compositions are zero, define $c_{S,r,t}=-1$. 

To check that $c$ is a cocycle, we need to show that $\partial(c): \{3-\text{faces}\} \to \k^\times$ sends each 3-face of the hypercube
to $1\in \k^\times$. The 3-faces of $\cal T(\cal V)$ are indexed by a subset $S$, and $r<t<u$ in $S^c$. The value assigned to this 3-face 
by $\partial(c)$ is the product $$\Pi=c_{S,r,t}c_{S,t,u}c_{S,r,u}c_{S\cup r,t,u}c_{S\cup t, r,u}c_{S\cup u, r,t}.$$ Note that the order $r<t<u$
determines a path $p=\delta_{S\cup r\cup t, u}\circ\delta_{S\cup r,t}\circ\delta_{S,r}$ from the vertex corresponding to $S$ to the
one corresponding to $S\cup r \cup t \cup u$. Multiplying by $c_{S,r,t}$ moves the path across the
appropriate 2-face and produces $c_{S,r,t}p=-\delta_{S\cup r\cup t, u}\circ\delta_{S\cup t, r}\circ\delta_{S,t}$. Continuing through all six 2-faces,
we eventually get back to the path $p$, so we have $\Pi p=p$, hence $\Pi=1$, as required. 

Since the cube is contractible, $c$ must be a coboundary, and this provides the
desired  scalar assignment. 

The uniqueness part of the lemma follows from the fact that the product of two such edge assignments
is a 1-cocycle. So it is the coboundary of some zero-cochain $\gamma: \cal T(\cal V) \to \k^\times$, which associates non-zero
scalars $\gamma(S)$ to each vertex $T_S$ of the hypercube. The isomorphism of the chain complexes is the map which is given by multiplication 
with $\gamma_S$ on $T_S$.
\end{proof}

\begin{remark}\label{rem:field}\rm{
In the proof of Lemmas \ref{lem:almostcomm} and \ref{lem:assign} above, we have worked over a field $\k$.  In order to work over 
$\bb Z$, or another commutative ring, one must  check that all the scalar obstructions to the commutativity of the hypercube are units.  
For signed graphical arrangements --- and in particular, as we will later see, for planar link projections --- it is straightforward to check that these scalar 
obstructions are always $\pm 1$ (this is not the case for non-graphical arrangements).  Thus in the graphical case,  we may take $\k = \bb Z$.  Then
the proofs of Lemmas \ref{lem:almostcomm} and \ref{lem:assign}, 
as well as the rest of the proofs in this section, go through without change. 
}
\end{remark}

We have defined a bi-graded chain complex $\cal T (\cal V)$. We will denote the resulting homology groups by $H_{Kh}^j(\cal V)$, which, 
for each $j$, is a graded $\k$-vector space. 

We will claim that $H_{Kh}^\bullet(\cal V)$ is a categorification of the framed Jones polynomial of arrangements, however for this statement
to make sense we have to generalize the notion of graded Euler characteristic to half-integer degrees, as follows:
$$\chi_q(H_{Kh}^\bullet(\cal V)):=\sum_{j, l\in \bb Z\cdot 1/2} (-1)^j \dim H^{j,l}_{Kh}(\cal V) \cdot q^l.$$
Although the formula is the same as before, one will encounter terms of $i=\sqrt{-1}$ and $\sqrt{q}$ arising from the half-integer degrees.

\begin{proposition}
The cohomology $H_{Kh}^\bullet(\cal V)$ categorifies the framed Jones polynomial (\ref{eq:Jones}) of hyperplane arrangements.
\end{proposition}

\begin{proof}
Straightforward from the definition of the grading and the chain groups. 
\end{proof}

The goal of the rest of this section is to prove various properties of $H_{Kh}^\bullet$. 
The following proposition is worth mentioning on its own right, but it will gain special importance in Section \ref{subsec:links} where it will
be used to show that $H_{Kh}^\bullet$ is well-defined as a link invariant.

\begin{proposition}\label{prop:smallissues}
The cohomology $H_{Kh}^i(\cal V)$ is invariant under permuting the vectors of $\cal V$, and under multiplying a vector $\nu_i$ in $\cal V$ by $-1$.
\end{proposition}

\begin{proof}
Let $\sigma(\cal V)$ be the permuted vector arrangement, and $n_i(\cal V)$ denote the arrangement where $\nu_i$ is replaced by $-\nu_i$.
Note that $\sigma(\cal V)$ has the same chain groups as $\cal V$, permuted. The scalar assignment $\epsilon$ is randomly chosen for a complex 
from the set of scalar assignments which work for it, and may be different for $\cal V$ and $\sigma(\cal V)$. 
Consider the map
$\phi: \cal T(\cal V) \to \cal T(\sigma(\cal V))$ which sends each chain group $T_S$ isomorphically to the corresponding $T_{\sigma(S)}$ 
in $\cal T(\sigma(\cal V))$. This
commutes with the differentials up to scalars, hence defines a 1-cocycle $c_1$ of the hypercube, just like in the proof of Lemma \ref{lem:assign}. Since
the cube is contractible, $c_1$ is the boundary of some 0-cochain $c_0: \{S \subseteq [n]\} \to \k$. This $c_0$ is the adjustment needed for $\phi$ to be a 
chain isomorphism.

The case of $n_i(\cal V)$ is similar: now the chain groups are exactly the same, and some differentials may get multiplied by $-1$. 
The same homological argument works.
\end{proof}

We are now ready to discuss some of the important properties of $\cal T (\cal V)$ and $H_{Kh}^\bullet(\cal V)$, namely Gale duality invariance,
long exact sequences for deletion-restriction, and a K\"unneth Theorem. 

\begin{proposition}\label{prop:GaleInv}
If $\cal V$ and $\cal V^\vee$ are Gale dual signed arrangements, then there is an isomorphism of chain complexes 
$$
	\cal T(\cal V) \cong \cal T(\cal V^\vee).
$$
\end{proposition}
\begin{proof}
For brevity, let us denote $T_S(\cal V)$ by just $T_S$, and $T_{S^\vee}(\cal V^\vee)$ by $T_{S^\vee}^\vee$. 
$S^\vee$ is different from $S$ in that the positive and negative signs are exchanged, hence
$\widetilde{S^{\vee}}=(\tilde{S}^c)^\vee=(\tilde{S}^\vee)^c$.
Observe that
$$T_{S^\vee}^\vee \cong \extp^\bullet H_{\widetilde{S^\vee}}^\vee \otimes \extp^\bullet W_{\widetilde{S^\vee}}^\vee 
\cong \extp^\bullet W_{\widetilde{S^{\vee}}^c} \otimes \extp^\bullet H_{\widetilde{S^{\vee}}^c}
\cong \extp^\bullet W_{\tilde{S}^\vee} \otimes \extp^\bullet H_{\tilde{S}^\vee}
\cong \extp^\bullet W_{\tilde{S}} \otimes \extp^\bullet H_{\tilde{S}}
\stackrel{\sigma}{\cong} T_S.$$ 
We claim that $\sigma$, which is the isomorphism of chain groups given by the switching of tensor factors, is actually an isomorphism of chain complexes.

Suppose $r \notin S$ and $r \in [n]_+$. Then $r \in [n]_-^\vee$
and $\nu_r \notin V_{\tilde{S}}$ if and only if $\nu_r^\vee \in V_{\widetilde{S^\vee}-r}$. So the Gale dual of a Type 1 differential is a Type 3
differential, and similarly the Gale dual of a Type 2 differential is a Type 4 differential.
Thus, using the same scalar assignments in the complexes $\cal T(\cal V)$ and $\cal T(\cal V^\vee)$, it follows that  $\sigma$ commutes 
with the differentials and that $\sigma$ is an isomorphism of chain complexes.  (If we used different scalar assignments in the differentials 
for $\cal T(\cal V)$ and $\cal T(\cal V^\vee)$, we would have to modify $\sigma$ accordingly in order to get a genuine chain map.)
\end{proof}

Now suppose that we are given a deletion restriction triple $\{\cal V, \cal V^{\nu_l}, \cal V-\nu_l\}$. The following 
theorem (which is a generalized analogue of the skein sequence for odd Khovanov homology) is similar to the long exact sequence
theorems of Section \ref{subsec:props}, with the difference that the sign of $\nu_l$ determines the direction of the maps.

\begin{theorem}\label{thm:lesKhov}
There is long exact sequence of homology groups, depending on the sign of $l$. If $l \in [n]_+$ and $\{\}$ denotes a shift of the $q$-grading, then
\begin{equation}\label{eq:lesplus}
 ...\to H_{Kh}^{i-1/2}(\cal V^{\nu_l})\{1/2\} \to H_{Kh}^i(\cal V) \to H_{Kh}^{i-1/2}(\cal V - \nu_l)\{1/2\} 
 \stackrel{\gamma_+}{\longrightarrow} H_{Kh}^{i+1/2}(\cal V^{\nu_l})\{1/2\} \to...
\end{equation}
where 
\begin{equation}\label{eq:gammaplus}
\gamma_+(x \otimes y)=\begin{cases}
                        x \otimes y & \text{ if } \nu_l \notin V_{\tilde S} \\
			x \otimes (\nu_l^\vee \wedge y) & \text{ if } \nu_l \in V_{\tilde S}.
                       \end{cases}
\end{equation}
If $l \in [n]_-$, then 
\begin{equation}\label{eq:lesminus}
...\to H_{Kh}^{i-1/2}(\cal V-\nu_l)\{1/2\} \to H_{Kh}^i(\cal V) \to H_{Kh}^{i-1/2}(\cal V^{\nu_l})\{1/2\} 
\stackrel{\gamma_-}{\longrightarrow} H_{Kh}^{i+1/2}(\cal V-\nu_l)\{1/2\} \to...
\end{equation}
where 
\begin{equation}\label{eq:gammaminus}
\gamma_-(x \otimes y)=\begin{cases}
                        (\nu_l \wedge x) \otimes y & \text{ if } \nu_l \notin V_{\tilde S-l} \\
			x \otimes y & \text{ if } \nu_l \in V_{\tilde S-l}.
                       \end{cases}
\end{equation}
\end{theorem}

\begin{proof}
 The proof is along the same lines as the proofs of the long exact sequence theorems of Section \ref{subsec:props}, we outline the positive case.
We need to show that there is a short exact sequence of chain complexes
$$0\longrightarrow \cal T(\cal V^{\nu_l})[1/2]\{1/2\}\stackrel{\iota}{\longrightarrow} \cal T(\cal V) \stackrel{\pi}{\longrightarrow} \cal T(\cal V -\nu_l)[1/2]\{1/2\}.$$
Note that, ignoring the grading for a moment, 
$$\cal T(\cal V^{\nu_l})=\bigoplus_{l \notin S} T_S^{\nu_l}=\bigoplus_{l\notin S} \extp^\bullet H_{\tilde{S}}^{\nu_l} \otimes \extp^\bullet W_{\tilde{S}}^{\nu_l}\cong
\bigoplus_{l\notin S} \extp^\bullet H_{\widetilde{S\cup l}} \otimes \extp^\bullet W_{\widetilde{S\cup l}},
$$
so we define $\iota$ to be the natural inclusion into $\cal T(\cal V)$. This increases the homological degree and the $q$-degree by $1/2$ each, hence the grading
shifts above. In turn, $\pi$ sends each $T_S$ for $l\notin S$ to itself in $\cal T(\cal V-\nu_l)$, and everything else to zero. Note that this is a map of bi-degree
$(-1/2,-1/2)$, hence with the grading shift as above it is of degree zero.

The maps $\iota$ and $\pi$ are chain maps up to possible scalar obstructions that may arise from different choices of scalar assignments in the three complexes.
This can be eliminated by applying a suitable scalar adjustment to $\iota$ and $\pi$.
The map $\gamma_+$ is the map arising from the snake lemma.
\end{proof}

The proof of the K\"unneth Theorem \ref{thm:kunneth} applies without any adjustment, so $H_{Kh}$ also satisfies a K\"unneth Theorem:

\begin{proposition}
 For two signed vector arrangements $\cal V$ and $\cal V'$ and the product arrangement $\cal V \times \cal V'$, there
is an isomorphism
$$H^\bullet_{Kh}(\cal V \times \cal V')\cong H^\bullet_{Kh}(\cal V) \otimes H^\bullet_{Kh}(\cal V').$$
\end{proposition}

\subsection{Reidemeister Invariance}\label{subsec:reidinv}
\begin{theorem}\label{thm:Reidemeister}
If $\cal V$ and $\cal V'$ are signed arrangements which differ by a framed Reidemeister move (as listed in Section \ref{subsec:SignedArrnments}), 
then $\cal T(\cal V)$ and $\cal T(\cal V')$ are chain homotopy equivalent.
\end{theorem}
\begin{proof}
By Proposition \ref{prop:DualPairs}, which states that Reidemeister moves come in Gale dual pairs, 
and Proposition \ref{prop:GaleInv}, which claims that $H_{Kh}$ is Gale duality--invariant,
it is enough to show Reidemeister invariance for one member of each Gale dual pair. 
We choose to prove $wR1^\vee$, $R2^\vee$ and $R3^\vee$.
Like \cite{OzsvathRasmussenSzabo}, we essentially follow the method and exposition of \cite[Section 3.5]{BarNatan}.

To show $wR1^\vee$, suppose that $l \in [n]_+$, $m \in [n]_-$, and $\nu_l^\vee=\nu_m^\vee=0$.
The chain complex $\cal T (\cal V)$ can be written as a direct sum of two faces of the cube, one for sets which contain $l$,
and one for those which do not:
$$\cal T (\cal V)=\bigoplus_{l \notin S} T_S \oplus \bigoplus_{l \notin S} T_{S \cup l}.$$
If $l \notin S$, then $\nu_l \notin V_{\tilde{S}}$: this is implied by the assumption that $\nu_l^\vee=0$. Hence 
$H_{\tilde{S}}\cong \k \nu_l \oplus H_{\tilde{S}\cup l}$. This means that
$$T_S=\extp^\bullet H_{\tilde{S}} \otimes \extp^\bullet W_{\tilde{S}}\cong 
(\extp^\bullet H_{\tilde{S}\cup l} \otimes \extp^\bullet W_{\tilde{S}}) \oplus (\nu_l \wedge (\extp^\bullet H_{\tilde{S}_\cup l}) \otimes \extp^\bullet W_{\tilde{S}}).$$
The differential $\delta_{S,l}$ is of Type 1, hence it is the identity on the first summand above, meaning that
$$\cal T'= \bigoplus_{l\notin S} \extp^\bullet H_{\tilde{S}\cup l} \otimes \extp^\bullet W_{\tilde{S}} \stackrel{\delta_{S,l}}{\longrightarrow} \bigoplus T_{S \cup l}$$
is an acyclic sub-complex, so the homology of $\cal T(\cal V)$ doesn't change if we factor out by $\cal T'$. After the factorization what remains is 
$\bigoplus_{l \notin S} \nu_l \wedge (\extp^\bullet H_{\tilde{S}_\cup l}) \otimes \extp^\bullet W_{\tilde{S}}$, which is isomorphic to $\cal T(\cal V^{\nu_l})$
up to a bi-degree shift of $(1/2,3/2)$, so $H^\bullet_{Kh}(\cal V^{\nu_l})\cong H^\bullet_{Kh}(\cal V)[1/2]\{3/2\}$.

We can similarly analyze the case of contracting $\nu_m\in [n]_-$. In this case for any $m\notin S$,
$$T_{S\cup m}=\extp^\bullet H_{\widetilde{S\cup m}} \otimes \extp^\bullet W_{\widetilde{S\cup m}} \cong
(\extp^\bullet H_{\tilde{S}} \otimes \extp^\bullet W_{\tilde{S}}) \oplus (\nu_m\wedge (\extp^\bullet H_{\tilde{S}}) \otimes \extp^\bullet W_{\tilde{S}}).$$
There is an acyclic sub-complex 
$$\bigoplus_{m\notin S} \extp^\bullet H_{\tilde{S}} \otimes \extp^\bullet W_{\tilde{S}} \stackrel{\delta_{S,m}}{\longrightarrow}
\nu_m\wedge (\extp^\bullet H_{\tilde{S}}) \otimes \extp^\bullet W_{\tilde{S}},$$
and factoring out by this leads to the isomorphism $H^\bullet_{Kh}(\cal V^{\nu_m})\cong H^\bullet_{Kh}(\cal V)[-1/2]\{-3/2\}$.

Thus, when contracting both $\nu_l$ and $\nu_m$, the two degree shifts cancel out and there is an isomorphism 
$H^\bullet_{Kh}(\cal V^{\{\nu_l,\nu_m\}})\cong H^\bullet_{Kh}(\cal V)$.

\parpic[l]{$\xymatrix{
  \bigoplus T_{S \cup m} \ar[r] & \bigoplus T_{S \cup l \cup m}  \\
  \bigoplus_{l,m \notin S} T_S \ar[r]^{w \otimes b} \ar[u] & \bigoplus T_{S \cup l} \ar[u]^{d \otimes d}
}$}
Moving on to $R2^\vee$, suppose that $l \in [n]_-$ and $m \in [n]_+$, $\nu_l^\vee=\alpha\nu_m^\vee$, $\alpha\neq 0$ and $\nu_l^\vee\neq 0$. 
We write the cube $\cal T(\cal V)$ as a direct sum of four faces, 
according to the incidence of $l$ and $m$ in $S$,
as shown on the left. In the bottom right corner, we observe that
$\nu_l, \nu_m \notin V_{\widetilde{S \cup l}}$, since otherwise there would be a linear dependency involving
only one of $\nu_l$ or $\nu_m$, contradicting $\nu_l^\vee = \alpha \nu_m^\vee$. (Keep in mind that 
$\widetilde{S \cup l}$ includes neither $l$ nor $m$.)

It follows that both $H_{\widetilde{S}}$ and $H_{\widetilde{S\cup l\cup m}}$ are strictly smaller than $H_{\widetilde{S\cup l}}$.
Therefore, all components of the bottom horizontal differential are of Type~4 (wedge by $\nu_l$ on the $H$ tensor factor and identity on the $W$ factor), 
and all components of the right vertical differential are of Type~1 (projection on the $H$ factor and identity on $W$).

{\bf Case 1.} Let us first assume that $\langle \nu_l,\nu_m \rangle \neq 0$ (we remark that this is always
the case for signed arrangements which arise from links or signed graphs). Note that $H_{\widetilde{S\cup l}}\cong H_{\widetilde{S}} \oplus \k\nu_l$,
so 
$$T_{S\cup l}\cong \extp^\bullet H_{\widetilde{S}} \otimes \extp^\bullet W_{\widetilde{S\cup l}}
\oplus (\nu_l \wedge \extp^\bullet H_{\widetilde{S}}) \otimes \extp^\bullet W_{\widetilde{S\cup l}}.$$
Since  $\langle \nu_l,\nu_m \rangle \neq 0$, the orthogonal projection of 
$H_{\widetilde{S\cup l}}$ to $H_{\widetilde{S\cup l\cup m}}$ is an isomorphism when restricted to
the $H_{\widetilde{S}}$, and 
there is an acyclic subcomplex 
$$\cal T'=\bigoplus \extp^ \bullet H_{\widetilde{S}} \otimes \extp^\bullet W_{\widetilde{S \cup l}}
\xrightarrow{\bigoplus \delta_{S \cup l,m}} \bigoplus T_{S \cup l \cup m}.$$

\parpic[r]{$\xymatrix{
  \bigoplus T_{S \cup m} \ar[r] & 0  \\
  \bigoplus_{l,m \notin S} T_S \ar[r] \ar[u] 
& \bigoplus (\nu_l \wedge \extp^\bullet\! H_{\widetilde{S}})\! \otimes\! \extp^\bullet W_{\widetilde{S \cup l}}\ar[u]\ar[lu]_\tau
}$}
Factoring out by $\cal T'$, we get the complex on the right. Note that the lower horizontal differential
(that is, wedge by $\nu_l$ on the fist tensor factor) is an isomorphism,
so it can be inverted and composed with the differential going up to produce a map $\tau$ from the 
lower right corner to the upper left corner. Now consider the sub-complex $\cal T''$ given by 
all elements $\alpha$ in the lower left corner, and pairs 
$$(\beta, \tau(\beta)) \in 
\bigoplus \left(\nu_l \wedge \extp^\bullet H_{\widetilde{S}}) \otimes \extp^\bullet W_{\widetilde{S \cup l}}\right)\oplus \bigoplus T_{S \cup m}.$$
This complex is acyclic due to the lower horizontal differential being an isomorphism.

Factoring out by $\cal T''$, we claim that the result is $\cal T'''\cong \bigoplus T_{S \cup m}$. To see this, note that factoring out by 
pairs $(\beta, \tau(\beta))$ identifies each element $(\beta,0)$ -- that is, $\beta$ in the bottom right corner and 0 in the top left -- with $(0, \tau(\beta))$. 
So while the free choice of $\alpha$ ``kills'' the bottom left corner, the free choice of $\beta$ identifies everything at the bottom right with something in the top left.
Hence, all that is left is a free choice on the top left.

Finally, observe that $\cal T'''\cong \bigoplus T_{S \cup m}$ is in turn isomorphic to $\cal T(\cal V^{\nu_l,\nu_m})$: as we saw in the proof of Theorem \ref{thm:lesKhov},
$\cal T(\cal V^{\nu_m})$ is a subcomplex (up to shifts) of $\cal T(\cal V)$, namely, the top row of the original square. In turn, as in the second half of
Theorem \ref{thm:lesKhov}, $\cal T(\cal V^{\nu_l,\nu_m})$ is a quotient of this by the top right corner of the original square, 
leaving the top left corner $\bigoplus T_{S \cup m}$ 
as needed, and the grading
shifts cancel with the global shift.

{\bf Case 2.} The case of $\langle \nu_l,\nu_m \rangle=0$ is similar if slightly more complicated. 
Note that in this case the vector $(\nu_l+\nu_m)$ is not orthogonal to either $\nu_l$
or $\nu_m$. Let $\bar{H}_S$ denote the orthogonal complement of $(\nu_l+\nu_m)$ in $H_{\widetilde{S\cup l}}$. 
We can write $H_{\widetilde{S\cup l}}\cong \bar{H}_S \oplus \k(\nu_l+\nu_m)$, and 
$$T_{S\cup l}\cong \extp^\bullet \bar{H}_S \otimes \extp^\bullet W_{\widetilde{S\cup l}}
\oplus ((\nu_l+\nu_m) \wedge \extp^\bullet \bar{H}_S) \otimes \extp^\bullet W_{\widetilde{S\cup l}}.$$
Again, the right vertical differential is an isomorphism when restricted to the first summand, and so there
is an acyclic subcomplex $\mathcal T'$ as before.

\parpic[r]{$\xymatrix{
  \bigoplus T_{S \cup m} \ar[r] & 0  \\
  \bigoplus_{l,m \notin S} T_S \ar[r] \ar[u] 
& \bigoplus ((\nu_l+\nu_m) \wedge \extp^\bullet\! \bar{H}_S)\! \otimes\! \extp^\bullet W_{\widetilde{S \cup l}}\ar[u]\ar[lu]_\tau
}$}
Factoring out by $\cal T'$, we get the complex on the right. Note that the lower horizontal differential on the first tensor factor of $T_S$
is the composition of wedging by $\nu_l$ folowed by the projection from 
$\extp^\bullet H_{\widetilde{S\cup l}}\cong \extp^\bullet \bar{H}_S\oplus ((\nu_l+\nu_m) \wedge \extp^\bullet\! \bar{H}_S) $ to its second component 
$(\nu_l+\nu_m) \wedge \extp^\bullet\! \bar{H}_S$.
Note that for any non-zero $x\in \extp^\bullet H_{\widetilde{S}}$ we have $\nu_l\wedge x \notin \extp^\bullet \bar{H}_S$, hence the composition
is injective and therefore an isomorphism. From here on the previous argument goes through, finishing the proof of the Reidemeister 2 invariance.

Recall that $R3^\vee$ states that if there are three distinct vectors $\nu_l, \nu_m$ and $\nu_p$ in $\cal V$ with $l,m \in [n]_-$ and $p \in [n]_+$, 
and a linear dependence $\nu_l^\vee=\alpha_m \nu_m^\vee+ \alpha_p \nu_p^\vee$
with non-zero coefficients, then $\cal V \leftrightarrow \cal V'$, where $\cal V'$ is the arrangement obtained from $\cal V^{\nu_l}$ 
by adding an extra vector $\nu_l'=\alpha_p \nu_m-\alpha_m\nu_p$ of
{\em positive} sign.

To prove $R3^\vee$ we will consider the complexes for both $\cal V$ and $\cal V'$ and reduce each one until we get isomorphic complexes.
Both $\cal T(\cal V)$ and $\cal T(\cal V')$ can be written as three dimensional cubes according to the incidence of $l$, $m$ and $p$
in $S$. We will first deal with the top faces of these cubes, which include the sets $S$ for which $l \in S$.

In the case of $\cal T(\cal V)$, note that $l\in [n]_-$, so $l\in S$ means $l \notin \widetilde{S}$. 
We can play the same game as in the proof of $R2^\vee$ above, with $m$ and $p$ playing the role of $l$ and $m$ in the proof of R2, respectively. 
Using the same steps as in the $R2$ proof, we can reduce $\cal T(\cal V)$
to the complex $\cal T'''(\cal V)$, as shown in the top row of Figure \ref{fig:R3proof}:
what remains is the bottom level of the original cube, two chain groups of the top level are gone, and everything in the remaining right corner
is identified via $\tau$ with something at the remaining left corner (as indicated by the ``$=$'' sign over the arrow). The figure shows
the outcome in the Case 1 scenario, Case 2 only differs in what exactly remains of $T_{S\cup l \cup m}$.

\parpic(2in,0.95in)[r]{$\xymatrix{
  \bigoplus T'_{S \cup p} \ar[r]^{b\otimes b} & \bigoplus T'_{S \cup m \cup p}  \\
  \displaystyle{\bigoplus_{l\in S; \; m,p \notin S} \!\!\!\! T'_S} \ar[r] \ar[u]^{d\otimes w} & \bigoplus T'_{S \cup m} \ar[u]
}$}
As for $\cal T(\cal V')= \bigoplus_S T'_S= \extp^\bullet H'_{\tilde{S}} \otimes \extp^\bullet W'_{\tilde{S}}$, 
there is a similar reduction process for the top level of the cube shown on the right. 
In particular, note that $\nu_l', \nu_m, \nu_p \in \widetilde{S\cup p}$,
and $W_{\widetilde{S \cup p}}$ is strictly larger than $W_{\widetilde{S}}$ or $W_{\widetilde{S\cup m \cup p}}$.
This means that the left and top differentials are Type 2 (wedge by $\nu_p^\vee$) and Type 3 (projection on the $W$ factor),
respectively.

\begin{figure}
$  \begin{array}{c} \xymatrix@R=3mm@C=3mm{
    & \bigoplus T_{S\cup l \cup p} \ar@{->}[rr]
        \ar@{<-}'[d][dd]^d
        \ar@{<-}[dr]^{=}
      & & 0 \ar@{<-}[dd] \\
    0 \ar@{->}[ur] \ar@{->}[rr] \ar@{<-}[dd] &
      & \bigoplus (\nu_m \wedge \extp^\bullet\! H_{\widetilde{S\cup l}})\! \otimes\! \extp^\bullet W_{\widetilde{S \cup l \cup m}} \ar@{->}[ur] \ar@{<-}[dd] \\
    & \bigoplus T_{S\cup p} \ar@{->}'[r][rr] & & \bigoplus T_{S\cup m\cup p} \\
    \displaystyle{\bigoplus_{l,m,p \notin S} T_S} \ar@{->}[rr] \ar@{->}[ur] & & \bigoplus T_{S\cup m} \ar@{->}[ur]
  } \end{array} $
  
   $\begin{array}{c} \xymatrix@R=3mm@C=3mm{
    & \bigoplus \extp^\bullet H'_{\widetilde{S\cup l \cup p}} \otimes (\nu_p^\vee \wedge \extp^\bullet W'_{\widetilde{S\cup l}}) \ar@{->}[rr]
        \ar@{<-}'[d][dd]
     & & 0 \ar@{<-}[dd] \\
    0 \ar@{->}[ur] \ar@{->}[rr] \ar@{<-}[dd] &
      & \bigoplus T'_{S\cup l \cup m} \ar@{->}[ur]
        \ar@{<-}[dd]
       \ar@{<-}[lu]^{=} \\
   & \bigoplus T'_{S\cup p} \ar@{->}'[r][rr] & & \bigoplus T_{S\cup m\cup p} \\
    \displaystyle{\bigoplus_{l,m,p \notin S} T_S} \ar@{->}[rr] \ar@{->}[ur] & & \bigoplus T_{S\cup m} \ar@{->}[ur]
  }\end{array} $
\caption{The reduced cubes $\cal T'''(\cal V)$ and $\cal T'''(\cal V')$ of the Reidemeister 3 proof.}\label{fig:R3proof}
\end{figure}
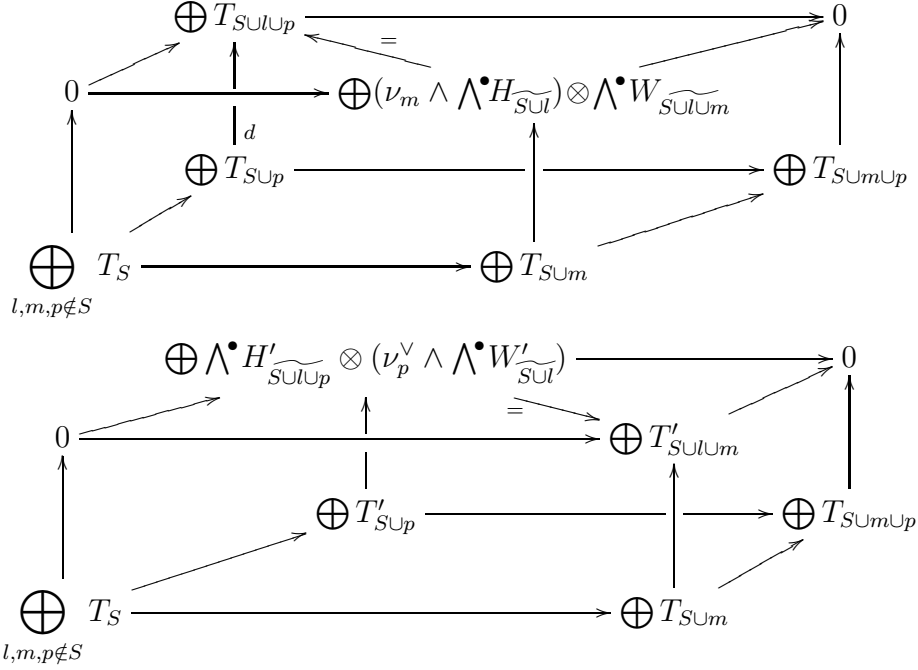

From here on we can follow the steps of the R2 proof, noting that since $\nu_l'=\alpha_p \nu_m-\alpha_m\nu_p$, we 
have $\langle \nu_m^\vee, \nu_p^\vee \rangle \neq 0$, and so we only need to deal with the easier Case 1. In the top left corner 
write $T'_{{S\cup p}} \cong \extp^\bullet H'_{\widetilde{S\cup p}} \otimes \extp^\bullet W'_{\widetilde{S}} \oplus 
\extp^\bullet H'_{\widetilde{S\cup p}} \otimes (\nu_p^\vee \wedge \extp^\bullet W'_{\widetilde{S}})$. 
(As before, $\nu_p^\vee$ really means the projection of 
$\nu_p^\vee$ onto $W'_{\widetilde{S\cup p}}$ .)
The top differential is an isomorphism when restricted to the first summand; factor out by the resulting acyclic subcomplex.
In the remaining complex the left vertical differential is an isomorphism so it can be inverted to produce a $\tau$ map
going diagonally down. The rest of the reduction is the same, and results in the cube shown in the lower row of Figure~\ref{fig:R3proof}.
In particular observe that on the top level everything in the remaining top left corner is identified via $\tau$ with something on
the right, as indicated by the ``$=$'' sign.

Our goal is to exhibit an isomorphism $\Phi: \cal T'''(\cal V) \to \cal T'''(\cal V')$. On the bottom level of the cubes, that is if $l \notin S$, we set
$\Phi: \extp^\bullet H_{\tilde{S}}\otimes \extp^\bullet W_{\tilde{S}} \to 
\extp^\bullet H'_{\tilde{S}}\otimes \extp^\bullet W'_{\tilde{S}}$ to be the restriction to $H_l$, which
is in these cases an isomorphism. On the top level $\Phi$ ``transposes'' the cube, identifyling $\bigoplus T_{S\cup l \cup p}$ with
$\bigoplus T'_{S\cup l \cup m}$ as follows: $H_{\widetilde{S\cup l \cup p}}$ and $H'_{\widetilde{S\cup l \cup m}}$ are both hyprplanes
of co-dimension one inside the othogonal complement of $\nu_l'$ in $H_{\widetilde{S\cup l \cup m}}$, and not orthogonal to each other
due to the assumption that $\nu_l^\vee=\alpha_m \nu_m^\vee+ \alpha_p \nu_p^\vee$. Hence, the orthogonal projection to 
$H'_{\widetilde{S\cup l \cup m}}$ restricts to an isomorphism $\Phi$ on $H_{\widetilde{S\cup l \cup p}}$. On the other hand,
$\nu_l^\vee=\alpha_m \nu_m^\vee+ \alpha_p \nu_p^\vee$ also implies that any linear dependency in $W_{\widetilde{S\cup l \cup p}}$
which involves $\nu_m$ and $\nu_p$ will in fact involve a scalar multiple of $\nu_l'$, hence $W_{\widetilde{S\cup l \cup p}}$ is
identical to $W'_{\widetilde{S\cup l \cup m}}$ (as subspaces of $\k^n$).

Having defined the isomorphism $\Phi$ on chain groups, we only have to check that it commutes with differentials and there
is anything to check only for the vertical ones. More precisely, one needs to verify that $\Phi \circ d= \tau \circ d \circ \Phi$ for
the differential $d$ in Figure \ref{fig:R3proof}, and similarly for the right front vertical differential.
This is a straightforward if slightly tedious exercise which we leave to the reader.
\end{proof}

\subsection{Signed arrangements from planar link projections and odd Khovanov homology}\label{subsec:links}
In this section we discuss the relationship of the framed odd Khovanov homology of signed arrangements introduced
in the preceding sections to the well-known Jones polynomial and odd Khovanov homology of links.

Let $D$ be a planar projection of an oriented link in the three-sphere. Our goal is to associate a signed vector arrangement
to $D$ by a two step process. First we will review the construction of the Tait graph
of $D$, a signed planar graph which encodes the link diagram. Then we associate a signed arrangement to the Tait graph by
the same method as in Section \ref{graphconst}.
We will refer to the vector arrangements constructed from planar link projections in this way as \emph{link vector arrangements}.  

\parpic[r]{\begin{picture}(0,0)%
\includegraphics{linkgraph.pstex}%
\end{picture}%
\setlength{\unitlength}{4144sp}%
\begingroup\makeatletter\ifx\SetFigFont\undefined%
\gdef\SetFigFont#1#2#3#4#5{%
  \reset@font\fontsize{#1}{#2pt}%
  \fontfamily{#3}\fontseries{#4}\fontshape{#5}%
  \selectfont}%
\fi\endgroup%
\begin{picture}(844,864)(5041,-2951)
\end{picture}%
}
To construct the Tait graph of a link diagram $D$, first suppose $D$ is connected as a planar graph.
Choose a checkerboard coloring of $D$, place vertices in each shaded
region, and draw an edge for each crossing, as illustrated in the picture on the right for a trefoil knot. 
Note that there are always two possible checkerboard shadings and the graphs corresponding to these
are planar duals of each other.

When the link diagram $D$ is not connected, the association of vertices is done component by component.  Thus, if $D$ is the disjoint union of 
two link diagrams, and we take the unbounded region to be shaded for both components, then this region will give rise to two vertices. Similarly, 
for planar graphs that are not connected we understand planar graph duality component-wise. In this way the double dual of a 
planar graph is the graph itself, and the planar graph associated to the opposite checkerboard coloring is the planar dual.

Of course, such a procedure does not distinguish between under and over crossings.  
In order to keep track of the under/over information, we associate to each crossing of a shaded oriented link diagram $D$ a ``shaded sign'', 
as shown in Figure \ref{fig:shadedsigns}.
The purpose of using shaded signs as opposed to the traditional crossing signs will become apparent later when we discuss
smoothings (see the proof of Proposition \ref{prop:JonesAndJones}).
Now, given any shaded link diagram, 
we can encode it as a signed graph, i.e. a graph with signs assigned to the edges. 

\begin{figure}[h]
 \input{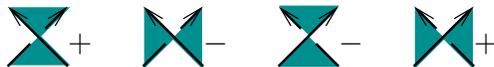}
\caption{Shaded sign conventions.}\label{fig:shadedsigns}
 \end{figure}

In turn, the procedure from Section \ref{graphconst} --- which was used there to associate an unsigned arrangement to an unsigned graph ---
extends in the obvious way to assign a signed arrangement to signed graphs. In this section we are going to use the reduced
construction (where one factors out by the line in the intersection of all the hyperplanes).

Let us continue by a brief discussion of framed links and Reidemeister moves. 
The Reidemeister 1 move most often used in knot theory allows for the un-twisting of a positive or negative ``kink'' 
(shown in Figure \ref{fig:R1}). The corresponding graph Reidemeister
moves (via the Tait graph construction) are the contraction of positive or negative leaf edges and the 
deletion of positive or negative loop edge, arising from opposite
checkerboard shadings, as shown in the top row of Figure \ref{fig:R1}. In accordance with Section \ref{subsec:SignedArrnments},
here we use a weak version of the Reidemeister 1 move (denoted $R1^f$ for ``framed''), 
shown in the bottom row of the same figure. Link diagrams modulo $R1^f$, $R2$ and $R3$ characterize {\em framed} links, where $R2$ and $R3$ are 
the usual Reidemeister 2 and 3 moves. Figure \ref{fig:R1} also illustrates 
how link Reidemeister moves give rise to Reidemeister moves of signed graphs. 
Note that the graph Reidemeister moves corresponding to the same link move but with
the opposite checkerboard shading are planar duals of each other (where planar duality of signed graphs switches the edge signs).
This in turn explains the Gale dual pairs of arrangement moves presented in Section \ref{subsec:SignedArrnments}.

\begin{figure}
 \input{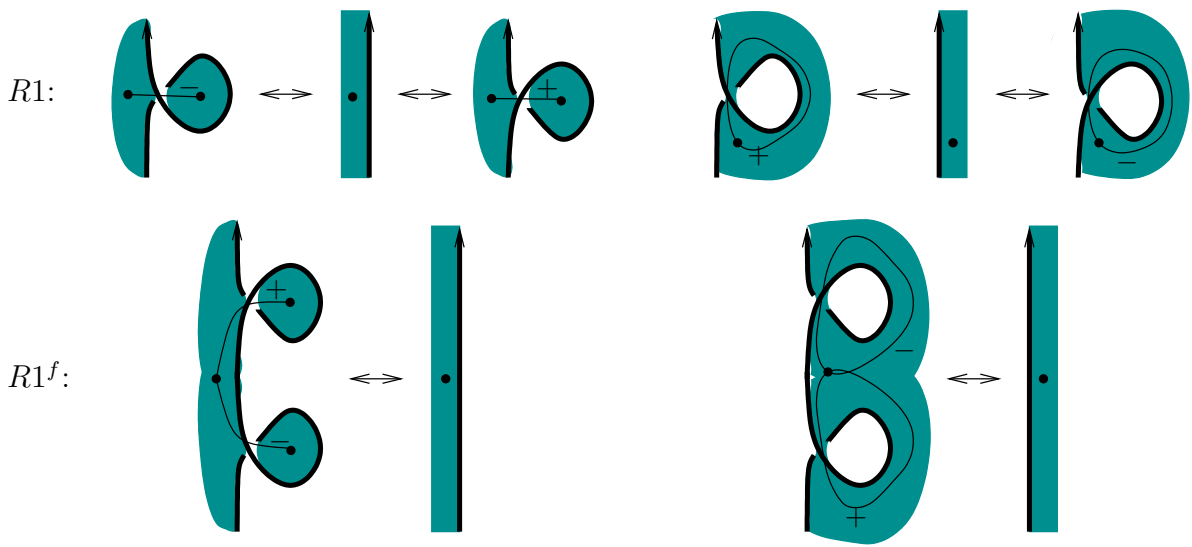}
 \caption{The Reidemeister 1 ($R1$) and framed Reidemeister 1 ($R1^f$) moves for links and signed graphs.}
 \label{fig:R1}
\end{figure}

\begin{remark}\label{rem:R1}{\rm
Note that the $wR1$ and $wR1^\vee$ moves of Section \ref{subsec:SignedArrnments}, 
when restricted to graphical or link arrangements, are slightly stronger than the $R1^f$ move shown in
Figure \ref{fig:R1}. For graphical arrangements $wR1$ and $wR1^\vee$ allow for the contraction or deletion 
of a pair of opposite sign leaf or loop edges, respectively, without requiring that the two 
leaf or loop edges be adjacent. In the language of links, this allows for the ``transfer of
framing from one link component to another'', without changing the total framing (writhe) of the link.
In other words, the link Reidemeister 1 move precisely equivalent to the 
$wR1$ and $wR1^\vee$ moves of Section \ref{subsec:SignedArrnments} allows for the simultaneous removal of 
a positive and a negative kink which need not be adjacent. Let us denote this move by $wR1$.
An example is shown in Figure \ref{fig:Framing}. 
} 
\end{remark}

\begin{figure}
 \input{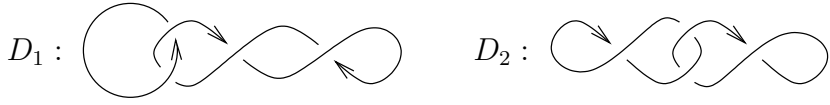}
 \caption{The links diagrams $D_1$ and $D_2$ are not framed Reidemeister equivalent: $D_1$ has two link components with 
 framings 0 and 2, while in $D_2$ the framings are 1 and 1, although the total framing in both 
 cases is $0+2=1+1=2$. However, $\cal V(D_1)$ and $\cal V(D_2)$ are Reidemeister equivalent 
 in the sense of Section \ref{subsec:SignedArrnments}. The reader can check that $D_1$ and $D_2$
 are indeed equivalent under the link $wR1$ move of Remark \ref{rem:R1}.}\label{fig:Framing}
\end{figure}

The following proposition follows from the definitions (or from the description of graph Reidemeister 
moves in \cite{BollobasRiordan}).
\begin{proposition}\label{prop:RmovesLinks}
The arrangement Reidemeister moves of Section \ref{subsec:SignedArrnments} preserve the class of link hyperplane arrangements.  When restricted to 
link vector arrangements, the arrangement Reidemeister moves agree with the $wR1$,
$R2$ and $R3$ Reidemeister moves of links. 
\end{proposition}

Recall that the well-known normalized Jones polynomial of an oriented link diagram $D$, an (unframed) isotopy invariant, 
can be computed by a state sum formula. Number the 
crossings of the link from 1 to $n$. To each crossing corresponds a ``0-smoothing'' and a ``1-smoothing'', as shown in Figure
\ref{fig:smoothings}.
To each $S \subseteq [n]$ corresponds a total smoothing of $D$, namely by 0-smoothing all the crossings not in $S$ and
1-smoothing all the crossings in $S$. A complete smoothing of $D$ is a disjoint union of a number of circles embedded in the plane. 
Let $c(S)$ denote the number of such circles in the $S$-smoothing. Let $n_0$ and $n_1$ denote the number of positive and negative
crossings of $D$, respectively (in the standard knot theory sense, not as ``shaded signs''). The normalized Jones polynomial is then
computed by the state sum formula
\begin{equation}\label{eq:UsualJones}
\mathcal J(D)=\sum_{S\subseteq[n]} (-1)^{|S|+n_1}q^{|S|+n_0-2n_1}(q+q^{-1})^{c(S)-1}.
\end{equation}

\begin{figure}[h]
 \input{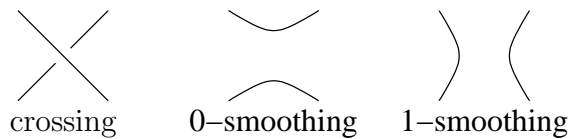}
\caption{The rule for smoothing a crossing.}\label{fig:smoothings}
 \end{figure}

\begin{proposition}\label{prop:JonesAndJones}
 Let $D$ be a link projection and let $\cal V(D)$ denote the corresponding signed vector arrangement (this involves some choices which
 are irrelevant by Proposition \ref{prop:LinkInv}).
 Then $J(\cal V(D))$ (of Equation (\ref{eq:Jones})) and $\cal J(D)$ (Equation (\ref{eq:UsualJones})) are related by the formula 
 $$J(\cal V(D))=\cal J(D)(-i\cdot q^{-3/2})^{w(D)},$$ where $w(D)=n_0-n_1$ 
 stands for the writhe of the link, a framed link invariant, and $i=\sqrt{-1}$. 
\end{proposition}

Before proving this proposition let us demonstrate it on a simple example.

\begin{example}\rm{
Let $L_0$ denote the unknot, with its projection to the plane as a simple circle, as shown in the figure below. 
The vector arrangement $\cal V(L_0)$ is -- independently of the checkerboard shading -- the empty arrangement in a zero dimensional vector space 
$\{\{0\};\emptyset\}$,
and $\cal J(L)=J(\calV(L))=1$, as expected.
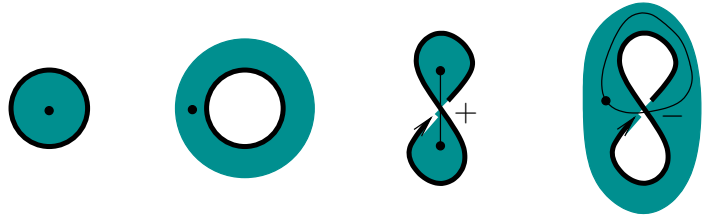
\begin{center}
 \input{figure8.pstex_t}
\end{center}
Let $L_1$ denote the ``positive kink'' as represented by the figure eight diagram shown in the figure above. $L_1$ differs from the unknot
only in a framing change, and its writhe is $w(L_1)=1$. The value of the Jones polynomial $\cal J(L_1)$ is still 1, and we expect that $J(\cal V(L_1))=-iq^{3/2}$.  
The associated vector arrangement $\cal V(L_1)$ depends on the choice of checkerboard shading for the diagram, as shown in the figure. If the two bounded
components are shaded, we obtain $\cal V(L_1)=\{\k; x_+\}$, a single positive vector in a vector space of dimension 1. If the unbounded component is shaded, we
get the Gale dual $\cal V(L_1)^\vee=\{\{0\};0_-\}$: a zero-dimensional vector space with a single zero vector of negative sign. The reader can check that 
feeding these two vector arrangements into the formula (\ref{eq:Jones}) results in the expected $-iq^{-3/2}$ in both cases.

The reader can check that the arrangements corresponding to the two checkerboard shadings of a negative kink $L_2$ are the same as above but with
the signs of the vectors switched; and that these both lead to $J(\cal V(L_2))=iq^{3/2}$.
}
\end{example}

{\it Proof of Proposition \ref{prop:JonesAndJones}.} 
The formulas (\ref{eq:Jones}) and (\ref{eq:UsualJones}) look similar, but differ in some normalizations (factors of $(-1)$ and $q$), and
more importantly, in the exponents of $q+q^{-1}$, which is $(\dim H_S+\dim W_S)$ in (\ref{eq:Jones}), but $(c(S)-1)$ in (\ref{eq:UsualJones}).
The main task is to show that indeed $\dim H_S+\dim W_S=c(S)-1$; showing that the difference in normalizations is canceled by the correction
term $(-i\cdot q^{-3/2})^{w(D)}$ is then a short routine check which we leave to the reader.

Let $D$ be an oriented link diagram with $n$ crossings numbered from 1 to $n$, 
and let us denote the total smoothing of $D$ arising from a subset $S\subseteq [n]$ by $D_S$. 
Choose a checkerboard shading of $D$ and let $G$ be the corresponding planar graph. Let $G_{\tilde{S}}$ denote the subgraph 
which contains only the edges of $G$ that are in $\tilde{S}=S_+\cup S_-^c$.
The key observation is that the lines of $D_S$ never ``cut across'' the edges of $G_{\tilde{S}}$. There are eight cases in checking this
claim: the four types of shaded crossings shown in Figure \ref{fig:shadedsigns}, and whether the crossing is in $S$ or not.
Two of these eight cases are shown in Figure \ref{fig:SmoothingEdge}, the rest are similar.
\begin{figure}[h!]
 \input{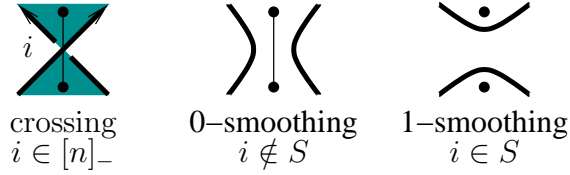}
 \caption{Smoothing lines don't cut across edges: the crossing $i$ is negative, hence if $i\notin S$ (i.e., the 0-smoothing is taken), then
 $i\in \tilde{S}$, hence the edge $i$ is in $G_{\tilde{S}}$. Similarly if $i \in S$ then the 1-smoothing is taken and $i \notin \tilde{S}$. }
 \label{fig:SmoothingEdge}
\end{figure}

If the checkerboard shading is so that the unbounded component is not shaded then this means that
the circles of $D_S$ ``follow the edges of $G_{\tilde{S}}$ on either side''.
Hence the circles of $D_S$
are in one-to-one correspondence with the connected components and the bounded faces of $G_{\tilde{S}}$: 
there is circle surrounding each connected component of $G_{\tilde{S}}$, as well as a circle inside each bounded face of
$G_{\tilde{S}}$. 

If the checkerboard shading has the unbounded component shaded, then it is not quite true that each connected component of $G_S$
has a circle surrounding it: the circle corresponding 
to one graph component is ``misplaced''. However, note that one can compactify the plane to obtain a sphere, then de-compactify by choosing the
infinity point to be in an un-shaded region. This reduces the case of a shaded unbounded region to the previous case of an unshaded unbounded region,
without changing the number of circles of $D_S$ or the number of components 
or faces of $G_{\tilde{S}}$. 

Recall from Section \ref{graphconst} that the dimension of $H_{\tilde{S}}$ is one less then the number of connected
components of $G_{\tilde{S}}$, and that the dimension of $W_{\tilde{S}}$ is the number of bounded faces of $G_{\tilde{S}}$. 
Hence $c(S)-1=\dim H_{\tilde{S}}+\dim W_{\tilde{S}}$, as needed.

Now the two state sum formulas only differ in normalizations and it is a quick check to verify that
the formula relating $J(\cal V(D))$ and $\cal J(D)$ holds.
\qed

Let us move on to show that the choices made in constructing the signed arrangement $\cal V(D)$ from $D$ are irrelevant
to $H_{Kh}(\cal V(D))$, and hence also to is graded Euler characteristic $J(\cal V(D))$.

\begin{proposition}\label{prop:LinkInv}
 Let $\cal V(D)$ denote the signed arrangement corresponding to a link diagram $D$. Then $H_{Kh}(\cal V(D))$
 does not depend on the choices made in the construction of $\cal V(D)$, and it is a framed link invariant. 
\end{proposition}

\begin{proof}
There choices made when associating a vector arrangement to a link diagram were the choice of checkerboard shading for $D$,
the ordering of the crossings, and the choice of edge orientations for the Tait graph in order to construct a vector arrangement
from it. Any of these choices lead to isomorphic chain complexes $\cal T(\cal V(D))$: the opposite checkerboard shading
gives rise to the Gale dual arrangement $\cal V(D)^\vee$, whose chain complex $\cal T(\cal V(D)^\vee)$ is isomorphic
to $\cal T(\cal V(D))$ by Theorem \ref{prop:GaleInv}.
A permutation of the crossings amounts to a permutation of the vectors, and changing the orientation of an edge of the Tait graph 
multiplies the corresponding vector in $\cal V(D)$ by $(-1)$. Both of these were shown to not change the isomorphism class of the chain complex in 
Proposition \ref{prop:smallissues}.

Finally, $H_{Kh}(\cal V(D))$ was shown to be arrangement Reidemeister invariant, hence by Proposition \ref{prop:RmovesLinks} it is
an invariant under the $wR1$, $R2$ and $R3$ moves of links, and also under $R1^f$ since it is weaker than $wR1$. Hence $H_{Kh}(\cal V(D))$
is a framed link invariant.
\end{proof}

The following proposition states that for a link arrangement $\cal V(L)$, $H^\bullet_{Kh}(\cal V(L))$ is 
isomorphic to the reduced odd Khovanov homology of $L$ up to a framing-dependent 
degree shift. Note that this is a categorified version of Proposition \ref{prop:JonesAndJones}.
\begin{proposition}\label{prop:OddKH}
Let $\cal V(L)$ be a link vector arrangement corresponding to a planar projection of a link $L$ in the three-sphere. Recall that $w(L)=n_0-n_1$ denotes the
writhe of the link, and let $Kh^\bullet(L)$ denote the reduced odd Khovanov homology of $L$. Let $[\cdot]$ and $\{\cdot\}$ denote shifts in the homological
and the $q$-grading, respectively. Then 
$$H^\bullet_{Kh}(\cal V(L))\cong Kh^\bullet(L)[-1/2\cdot w(L)]\{-3/2\cdot w(L)\}.$$
\end{proposition}
\begin{proof}
Note that while the chain groups $\cal T(\cal V(L)$ agree
with those of \cite{OzsvathRasmussenSzabo} up to the grading shifts above, the boundary maps 
we use differ slightly from the 
\cite{OzsvathRasmussenSzabo} boundary maps, which are not Gale-duality invariant. One can check this, for example, in the case of the Hopf link. 
However, in \cite{Bloom}, Bloom gives an alternative (slightly more symmetric) definition of a chain complex associated 
to a connected link projection, and he proves that the resulting homology is isomorphic to odd Khovanov homology.  It is a straightforward, if slightly tedious, 
combinatorial exercise (which we leave to the reader) to check that our boundary maps for a link vector arrangement agree with Bloom's, this
implies that $H^\bullet_{Kh}(\cal V)$ coincides with reduced odd Khovanov homology for connected diagrams. In turn, Reidemeister invariance implies that they
coincide for all links: any link diagram can be made connected by performing Reidemeister 2 moves.
\end{proof}

\begin{remark}\rm{We end with a brief note on Conway mutation and 2-isomorphisms.  We have seen that 
$H^\bullet_{Kh}(\cal V)$ of an arrangement $\cal V$ corresponding to a link is isomorphic to 
the reduced odd Khovanov homology the link, hence it is mutation invariant \cite{Bloom}. 

A 2-isomorphism of graphs is a cycle-preserving bijection between 
their edge sets. In \cite{Greene} Greene proves that two reduced,
alternating link diagrams are Conway mutants if and only if the corresponding (un-signed) graphs are 2-isomorphic, and uses this to show that 
the Heegaard Floer homology of the branched double cover provides a complete invariant for the mutation type of alternating links. 

Note that for un-signed (not necessarily
planar) graphical arrangements $\cal V$, $H^\bullet_{Kh}(\cal V)$ is a 2-isomorphism invariant by definition. 
Through \cite{Greene} this provides another proof of its mutation invariance 
for alternating links. }
\end{remark}

\end{document}